\documentclass[journal,twoside,web]{ieeecolor}

\usepackage{generic}
\usepackage{cite}
\usepackage{amsmath,amssymb,amsfonts,amsthm,balance,bm,mathtools}
\usepackage{algorithm,algpseudocode,algorithmicx}
\usepackage{graphicx}
\usepackage{textcomp}
\usepackage{subcaption}
\usepackage{hyperref}

\newtheorem{theorem}{Theorem}
\newtheorem{lemma}{Lemma}
\newtheorem{proposition}{Proposition}
\newtheorem{definition}{Definition}

\newtheorem{remark}{Remark}

\usepackage{tikz}
\usetikzlibrary{calc,trees,positioning,arrows,chains,shapes.geometric,decorations.pathreplacing,decorations.pathmorphing,shapes, matrix,shapes.symbols}

\DeclareMathOperator*{\spec}{spectrum}

\newcommand{\differential}{{\rm{d}}}
\newcommand{\tr}{{\rm{trace}}}

\newcommand{\RNum}[1]{\uppercase\expandafter{\romannumeral #1\relax}}

\def\BibTeX{{\rm B\kern-.05em{\sc i\kern-.025em b}\kern-.08em
    T\kern-.1667em\lower.7ex\hbox{E}\kern-.125emX}}

\allowdisplaybreaks

\title{
Fixed Horizon Linear Quadratic Covariance Steering in Continuous Time with Hilbert-Schmidt Terminal Cost
}

\author{Tushar Sial and Abhishek Halder, \IEEEmembership{Senior Member, IEEE}
\thanks{Tushar Sial and Abhishek Halder are with the Department of Aerospace Engineering, Iowa State University, Ames, IA 50011, USA, {\tt\small{\{tsial,ahalder\}@iastate.edu}}.%
}
\thanks{This research was supported by NSF award 2111688.}
}

\begin{document}
\maketitle

\begin{abstract}
We formulate and solve the fixed horizon linear quadratic covariance steering problem in continuous time with a terminal cost measured in Hilbert-Schmidt (i.e., Frobenius) norm error between the desired and the controlled terminal covariances. For this problem, the necessary conditions of optimality become a coupled matrix ODE two-point boundary value problem. To solve this system of equations, we design a matricial recursive algorithm and prove its convergence. The proposed algorithm and its analysis make use of the linear fractional transforms parameterized by the state transition matrix of the associated Hamiltonian matrix. To illustrate the results, we provide two numerical examples: one with a two dimensional and another with a six dimensional state space.
\end{abstract}



\section{Introduction}\label{sec:introduction}
We consider the following continuous-time linear quadratic (LQ) covariance steering problem:
\begin{subequations}
\begin{align}
    &\inf_{\bm{u}_{t}\in \mathcal{U}}\:\phi\left(\bm{\Sigma}_1,\bm{\Sigma}_{d}\right)+\!\int_{t_0}^{t_1}\!\!\mathbb{E}\left(\| \bm{u}_{t} \|_2^2 + \bm{x}_t^{\top}\bm{Q}_{t}\bm{x}_t \right)\differential t  \label{OCPobjective}\\
    &\text{subject to}\nonumber\\
    &\differential\bm{x}_t = \bm{A}_t\bm{x}_t\:\differential t + \bm{B}_t\bm{u}_t\:\differential t + \bm{B}_t\:\differential\bm{w}_t, \label{OCPdynamics}\\
    &\bm{x}_{0}:=\bm{x}(t=t_0) \sim \rho_{0}:=\mathcal{N}\left(\bm{0},\bm{\Sigma}_0\right), \rho_{d} := \mathcal{N}\left(\bm{0},\bm{\Sigma}_d\right),\label{OCPinitialANDdesiredPDFs}
\end{align}
\label{OCP}
\end{subequations}
where the time horizon is finite and fixed as $[t_0,t_1]$, the symbol $\mathbb{E}$ denotes the expectation w.r.t. the state vector $\bm{x}_{t}\in\mathbb{R}^{n}$, the $\mathcal{N}$ denotes multivariate normal probability density function (PDF), and $\sim$ is a shorthand for ``follows the distribution". In \eqref{OCPinitialANDdesiredPDFs}, the $\rho_0$ is the PDF for the initial state $\bm{x}_0$, and the $\rho_d$ is the desired PDF for the terminal state at $t=t_1$. Depending on the choice of control policy, the statistics of the controlled terminal state $\bm{x}_{1}:=\bm{x}(t=t_1)$ may not match with the desired terminal statistics $\rho_d$.

In \eqref{OCPobjective}, $\bm{\Sigma}_{1}$ is the covariance of the terminal state $\bm{x}_{1}:=\bm{x}(t=t_1)$, which depends on the choice of control policy. The terminal cost $\phi$ is a measure of error between the terminal state covariance $\bm{\Sigma}_{1}$ and a desired state covariance $\bm{\Sigma}_{d}$. While many choices for $\phi$ are possible, in this work, we fix $\phi$ to be the squared Frobenius or Hilbert-Schmidt norm:
\begin{align}
&\phi\left(\bm{\Sigma}_1,\bm{\Sigma}_{d}\right) := \frac{1}{2}\|\bm{\Sigma}_1-\bm{\Sigma}_{d}\|_{\mathrm{Frobenius}}^{2}\nonumber\\
&= \frac{1}{2}\langle\bm{\Sigma}_1-\bm{\Sigma}_{d},\bm{\Sigma}_1-\bm{\Sigma}_{d}\rangle=\frac{1}{2}\tr\left(\bm{\Sigma}_1-\bm{\Sigma}_{d}\right)^2,
\label{DefTerminalCost}
\end{align}
wherein the Hilbert-Schmidt inner product $\langle\bm{M},\bm{N}\rangle := \tr\left(\bm{M}^{\top}\bm{N}\right)$. The prior work \cite{halder2016finite} considered LQ covariance steering with $\phi$ as the squared Bures-Wasserstein distance metric \cite{bhatia2019bures} between $\bm{\Sigma}_1, \bm{\Sigma}_{d}$, given by
\begin{align}
\frac{1}{2}\tr\left(\bm{\Sigma}_1+\bm{\Sigma}_{d}-2\left(\bm{\Sigma}_1^{\frac{1}{2}}\bm{\Sigma}_{d}\bm{\Sigma}_1^{\frac{1}{2}}\right)^{\frac{1}{2}}\right),
\label{squaredBuresWasserstein}    
\end{align}
where the exponent $\frac{1}{2}$ denotes principal square root. The metric \eqref{squaredBuresWasserstein} originates from the theory of optimal transport \cite{givens1984class,gelbrich1990formula}, and coincides with \eqref{DefTerminalCost} when $\bm{\Sigma}_1,\bm{\Sigma}_d$ commute (e.g., when the state dimension $n=1$).

Intuitively, the terminal cost $\phi$ penalizes the mismatch between the second-order statistics of the controlled terminal state $\bm{x}_{1}$ from that of a desired state. In this sense, $\bm{\Sigma}_{d}$ can be interpreted as a desired statistical tolerance for control performance at $t=t_1$.

The set of admissible control inputs 
\begin{align*}
\mathcal{U}:=\bigg\{&\bm{u}:[t_0,t_1]\times\mathbb{R}^{n}\mapsto\mathbb{R}^{m}\mid \mathbb{E}_{\rho}\bigg\{\!\int_{t_0}^{t_1}\!\|\bm{u}\|_2^2\differential t\bigg\}<\infty\\
&\forall\rho:\mathbb{R}^{n}\mapsto\mathbb{R}_{\geq 0}, \int_{\mathbb{R}^{n}}\rho\:\differential\bm{x} = 1, \int_{\mathbb{R}^{n}}\|\bm{x}\|_2^2\rho\:\differential\bm{x} < \infty\bigg\}
\end{align*} 
in \eqref{OCPobjective} comprises adapted finite energy control policies such that \eqref{OCPdynamics} has a strong solution.

In the controlled It\^{o} diffusion \eqref{OCPdynamics}, the standard Wiener process $\bm{w}_t\in\mathbb{R}^{m}$. We take the same $\bm{B}_t$ as the control and the noise coefficients, i.e., assume that the input and the noise channels are identical. This is indeed common in practice, e.g., when the noise is due to stochastic actuation, or when the noise enters through external force/torque/current.

In \eqref{OCPinitialANDdesiredPDFs}, the initial state PDF $\rho_0$ and the desired terminal sate PDF $\rho_{d}$ are centered multivariate normals with prescribed covariances $\bm{\Sigma}_0,\bm{\Sigma}_{d}$, respectively.

With $\Phi$ as in \eqref{DefTerminalCost}, we make the following assumptions about the data for problem \eqref{OCP}.
\begin{itemize}
\item[\textbf{A1.}] The pair $\left(\bm{A}_t,\bm{B}_{t}\right)$ is bounded and continuous w.r.t. $t\in[t_0,t_1]$, and is uniformly controllable\footnote{i.e., the controllability Gramian is strictly positive definite $\forall [s,t]\subseteq [t_0,t_1]$.}. 

\item[\textbf{A2.}] The state cost-to-go weights $\bm{Q}_{t}\succeq\bm{0}$ for all $t\in[t_0,t_1]$.

\item[\textbf{A3.}] The covariances $\bm{\Sigma}_{0},\bm{\Sigma}_{d}\succ\bm{0}$.
\end{itemize}

We wish to design a linear feedback controller 
\begin{align}
\bm{u}_{t} = \bm{K}_{t}\bm{x}_t 
\label{LinearController}    
\end{align}
with deterministic time-varying gain $\bm{K}_{t}:[t_0,t_1]\mapsto\mathbb{R}^{m\times n}$, so as to steer the state covariance from $\bm{\Sigma}_{0}$ at $t=t_0$, to a neighborhood of $\bm{\Sigma}_{d}$ at $t=t_1$, that is optimal w.r.t. \eqref{OCP}. 

We clarify here that the zero mean assumptions for $\rho_0,\rho_{d}$ are without loss of generality, since non-zero means can be tackled as usual by adding a time-varying affine term to \eqref{LinearController}, and separately designing a controller for the same (see Remark \ref{Remark:NonzeroMean} in Sec. \ref{sec:MainIdeas}).

\subsubsection*{Motivation} Historically, the steering of state covariance\footnote{i.e., the covariance of the controlled state $\bm{x}_t\in\mathbb{R}^{n}$ at time $t\in[t_0,t_1]$} $\bm{\Sigma}_t$ toward a desired covariance $\bm{\Sigma}_d$ has been motivated \cite{hotz1987covariance} by interpreting $\bm{\Sigma}_d$ as a specification of acceptable statistical performance. For example, aerospace guidance and navigation problems with desired performance are often formulated as covariance steering \cite{zhu1995covariance,ridderhof2020chance,benedikter2022convex,kumagai2024sequential}.

The covariance $\bm{\Sigma}_0$ captures the initial condition uncertainty, and is usually available from an estimator. A terminal cost $\phi\left(\bm{\Sigma}_1,\bm{\Sigma}_d\right)$ allows the control designer to trade off the cost for terminal second-order statistics mismatch with the average cost-to-go $\!\int_{t_0}^{t_1}\mathbb{E}\left(\| \bm{u}_{t} \|_2^2 + \bm{x}_t^{\top}\bm{Q}_{t}\bm{x}_t \right)\differential t$.

Our choice of the specific terminal cost \eqref{DefTerminalCost} is motivated by that it is the simplest, among other choices such as \eqref{squaredBuresWasserstein}, which is the non-commutative version of \eqref{DefTerminalCost}. While the work in \cite{halder2016finite} derived the conditions for optimality with terminal cost \eqref{squaredBuresWasserstein}, a custom numerical algorithm remains unavailable in that more general setting. By considering the simpler terminal cost \eqref{DefTerminalCost}, this work designs a provably convergent numerical algorithm for solving problem \eqref{OCP}. 

From a geometric viewpoint, the terminal cost \eqref{DefTerminalCost} is Euclidean in the sense that it ignores the Riemannian structure \cite[Ch. 6]{bhatia2009positive} of the cone of positive definite covariance matrices. Even so, our developments here reveal considerable subtlety in the design and analysis of the algorithm, and advances the state-of-the-art in numerically solving this class of problems.

\subsubsection*{Related works on covariance steering with terminal cost} In the \emph{discrete-time} case, there is a growing literature \cite{balci2020covariance,balci2021convexity,yin2022trajectory,saravanos2024distributed,morimoto2024minimum,nakashima2025formation} on fixed horizon LQ covariance steering problems with terminal cost. For the terminal cost $\Phi$, these works use the squared Bures-Wasserstein metric \eqref{squaredBuresWasserstein} or its unregistered version, the Gromov-Wasserstein metric \cite{memoli2011gromov,salmona2022gromov}.

In the \emph{continuous-time} case, fixed horizon covariance steering with terminal cost remains relatively underexplored. The first such work \cite{halder2016finite} considered fixed horizon LQ covariance steering with the Bures-Wasserstein terminal cost and derived the conditions for optimality for the same. The computation therein was done by shooting method, but a more principled algorithm to solve the same is not available. More recent work in \cite{hoshino2023finite} derives the conditions of optimality for the general nonlinear non-Gaussian version of this problem with the Wasserstein terminal cost, in the form of forward-backward stochastic differential equations. However, this work does not investigate the numerical solution.

The main reason behind this imbalance of literature between the discrete-time and continuous-time formulations is computational difficulty. The discrete-time LQ covariance steering with terminal cost naturally leads to semidefinite programming, hence is amenable to off-the-shelf interior point solvers. In contrast, the continuous-time LQ covariance steering with terminal cost leads to a coupled nonlinear system of matricial ODEs, and it is not obvious what is the principled algorithmic approach to solve the same.

\subsubsection*{Contributions} The specific contributions of this work are threefold.
\begin{itemize}
    \item For the fixed horizon LQ covariance steering problem with Hilbert-Schmidt terminal cost, we derive the conditions of optimality as a matricial boundary value problem in state covariance and its costate. 

    \item Informed by the structure of the coupled matrix equations resulting from the aforesaid conditions of optimality, we propose a symmetric matrix-valued recursive algorithm to numerically solve the same. One step of the proposed matricial recursion is a composition of four maps: one of them is affine, two are linear fractional transforms associated with two Riccati matrix ODEs, and the remaining is an analytically solvable stabilizing solution of a special continuous-time algebraic Riccati equation.  

    \item We prove that the proposed nonlinear matrix-valued recursion is convergent to its unique fixed point. We demonstrate the proposed method on two numerical examples: a noisy double integrator and a close-proximity orbital rendezvous.
\end{itemize}

\subsubsection*{Notational preliminaries} Most notations are introduced in situ. Throughout, we use $\otimes, \odot$ to denote the Kronecker product and the Hadamard (i.e., element-wise) product, respectively. As standard, all vectors are column vectors by default; row vectors are obtained by transposition, denoted as superscript $^{\top}$. The symbols ${\mathrm{vec}}, {\mathrm{vech}}, {\mathrm{diag}}$ are used to denote the vectorization (for matrix argument), the half-vectorization (for symmetric matrix argument), and the diagonal matrix operator (for vector argument), respectively. We use $\differential$ and ${\mathrm{D}}$ for the differential and the Jacobian, respectively. For computing the derivative of
a matrix-to-matrix map $\bm{F}(\bm{X})$, we utilize the Jacobian identification rule \cite[Ch. 9, Table 9.2]{magnus2019matrix}: ${\mathrm{vec}}\left(\differential\bm{F}(\bm{X})\right)={\mathrm{D}}\bm{F}(\bm{X})\differential\:{\mathrm{vec}}(\bm{X})$. For any random vector $\bm{x}$, its covariance matrix $\bm{\Sigma}:=\mathbb{E}\left[\left(\bm{x}-\mathbb{E}\left[\bm{x}\right]\right)\left(\bm{x}-\mathbb{E}\left[\bm{x}\right]\right)^{\top}\right]$ where the expectation operator $\mathbb{E}$ is taken w.r.t. the law or probability distribution of $\bm{x}$.

\subsubsection*{Organization} In Sec. \ref{sec:MainIdeas}, we present the conditions for optimality (Proposition \ref{prop:FOOC}) in the form of a matricial boundary value problem in state covariance and its costate. The equations include a forward-in-time Lyapunov matrix ODE and a backward-in-time Riccati matrix ODE. After a change of variable, this Lyapunov-Riccati system transforms into a Riccati-Riccati system of ODE boundary value problem. To solve this system of equations, in Sec. \ref{sec:Mappings}, we construct three matrix-valued mappings and analyze their properties. In Sec. \ref{sec:RecrsiveAlgorithm}, we use these mappings to propose a recursive algorithm for solving the conditions for optimality, and establish convergence guarantee (Theorem \ref{thm:mainresult}) for the proposed algorithm. Sec. \ref{sec:examples} details two numerical examples to illustrate the proposed method. Sec. \ref{sec:conclusions} concludes the paper. Auxiliary lemmas needed to prove the main results are collected in Appendix \ref{App:LemmaF4bound}.

\section{Conditions for Optimality}\label{sec:MainIdeas}
From \eqref{OCPdynamics} and \eqref{LinearController}, the controlled state $\bm{x}_{t}\sim\mathcal{N}\left(\bm{0},\bm{\Sigma}_{t}\right)$ with covariance dynamics
\begin{align}
\dot{\bm{\Sigma}}_{t} = \left(\bm{A}_t + \bm{B}_t\bm{K}_t\right)\bm{\Sigma}_t + \bm{\Sigma}_t\left(\bm{A}_t + \bm{B}_t\bm{K}_t\right)^{\!\top} \!\!+ \bm{B}_t\bm{B}_t^{\top}
\label{CovarianceDynamics}    
\end{align}
on $\mathbb{S}^{n}_{++}$, the cone of positive definite matrices. Thanks to the quadratic objective \eqref{OCPobjective}, we can then view $\bm{\Sigma}_t\in\mathbb{S}^{n}_{++}$ itself as the state, and define $\bm{P}_{t}\in\mathbb{S}^{n}$ (the affine set of $n\times n$ real symmetric matrices) as the costate. Specifically, we have the following conditions for optimality.

\begin{proposition}\label{prop:FOOC}
Consider problem \eqref{OCP} with assumptions \textbf{A1}-\textbf{A3}, and linear feedback \eqref{LinearController}. The necessary condition for optimality is a coupled ODE boundary value problem on the cotangent bundle $\mathcal{T}^{*}\mathbb{S}^{n}_{++}=\mathbb{S}^{n}_{++}\times\mathbb{S}^{n}$ in unknown $\left(\bm{\Sigma}_t^{\mathrm{opt}},\bm{P}_t^{\mathrm{opt}}\right)$, given by
\begin{subequations}
\begin{align}
\dot{\bm{\Sigma}}_{t}^{\mathrm{opt}} &= \left(\bm{A}_t - \bm{B}_t \bm{B}_t^{\top}\bm{P}_t^{\mathrm{opt}}\right)\bm{\Sigma}_t^{\mathrm{opt}} \nonumber\\
&+ \bm{\Sigma}_t^{\mathrm{opt}}\left(\bm{A}_t - \bm{B}_t\bm{B}_t^{\top}\bm{P}_t^{\mathrm{opt}}\right)^{\!\top}+ \bm{B}_t\bm{B}_t^{\top}, \label{CovarianceStateODE}\\
-\dot{\bm{P}}_{t}^{\mathrm{opt}} &= \bm{A}_{t}^{\top}\bm{P}_{t}^{\mathrm{opt}} +\bm{P}_{t}^{\mathrm{opt}}\bm{A}_{t} - \bm{P}_t^{\mathrm{opt}}\bm{B}_{t}\bm{B}_{t}^{\top}\bm{P}_{t}^{\mathrm{opt}} + \bm{Q}_{t}, \label{CovarianceCostateODE}\\
\bm{\Sigma}_0&\succ\bm{0} \quad\text{given}, \label{InitialCovariance}\\
\bm{P}_{1}^{\mathrm{opt}} &= \bm{\Sigma}_{1}^{\mathrm{opt}}-\bm{\Sigma}_{d}, \quad \bm{\Sigma}_d\succ\bm{0} \quad\text{given}. \label{TransveersalityCondition}
\end{align}
\label{FOOC}
\end{subequations}
The associated optimal gain 
\begin{align}
\bm{K}_{t}^{\mathrm{opt}} = -\bm{B}_{t}^{\top}\bm{P}_{t}^{\mathrm{opt}},
\label{OptimalGain}    
\end{align}
and the optimal control $\bm{u}_{t}^{\mathrm{opt}} = \bm{K}_{t}^{\mathrm{opt}}\bm{x}_t$.
\end{proposition}
\begin{proof}
Using \eqref{LinearController} and the invariance of trace under cyclic permutation, we write the Hamiltonian $H:\mathcal{T}^{*}\mathbb{S}^{n}_{++}\times\mathbb{R}^{m\times n}\mapsto\mathbb{R}$ for problem \eqref{OCP} as
\begin{align}
&H\left(\bm{\Sigma}_{t},\bm{P}_{t},\bm{K}_{t}\right) = \langle\bm{\Sigma}_{t},\bm{K}_{t}^{\top}\bm{K}_{t} + \bm{Q}_{t}\rangle\nonumber\\
&\!\!+ \langle\bm{P}_{t},  \left(\bm{A}_t + \bm{B}_t\bm{K}_t\right)\bm{\Sigma}_t + \bm{\Sigma}_t\left(\bm{A}_t + \bm{B}_t\bm{K}_t\right)^{\!\top} \!\!+ \bm{B}_t\bm{B}_t^{\top}\rangle.
\label{defHamiltonian}    
\end{align}
Applying Pontryagin's minimum principle to the Hamiltonian \eqref{defHamiltonian}, at optimality we have
\begin{subequations}
\begin{align}
\dot{\bm{\Sigma}}_{t}^{\mathrm{opt}} &= \dfrac{\partial H}{\partial \bm{P}_t^{\mathrm{opt}}},\label{StateODE}\\
\dot{\bm{P}}_{t}^{\mathrm{opt}} &= -\dfrac{\partial H}{\partial \bm{\Sigma}_t^{\mathrm{opt}}},\label{CostateODE}\\
\bm{0} &= \dfrac{\partial H}{\partial \bm{K}_{t}^{\mathrm{opt}}} = \left(\bm{K}_{t}^{\mathrm{opt}}+\bm{B}_{t}^{\top}\bm{P}_{t}^{\mathrm{opt}}\right)\bm{\Sigma}_{t}^{\mathrm{opt}}.\label{Gain}
\end{align}
\label{PMP}
\end{subequations}
In particular, since \eqref{Gain} must hold for arbitrary $\bm{\Sigma}_{0}\succ\bm{0}$, and thus for arbitrary $\bm{\Sigma}_{t}^{\mathrm{opt}}\succ\bm{0}$, we obtain \eqref{OptimalGain}. 

The state ODE \eqref{StateODE} is the same as \eqref{CovarianceDynamics}, which combined with \eqref{OptimalGain}, yields \eqref{CovarianceStateODE}. Likewise, the costate ODE \eqref{CostateODE} combined with \eqref{OptimalGain}, yields \eqref{CovarianceCostateODE}. We obtain \eqref{TransveersalityCondition} using the transversality condition
$$\bm{P}_{1}^{\mathrm{opt}} = \dfrac{\partial\phi}{\partial\bm{\Sigma}_{1}^{\mathrm{opt}}}.$$ 
This completes the proof.
\end{proof} 

\begin{remark}\label{Remark:NonzeroMean}
It is straightforward to generalize Proposition \ref{prop:FOOC} for $\rho_{0}=\mathcal{N}(\bm{\mu}_{0},\bm{\Sigma}_{0})$, $\rho_{d}=\mathcal{N}(\bm{\mu}_{d},\bm{\Sigma}_{d})$ with given nonzero means $\bm{\mu}_0,\bm{\mu}_{d}\in\mathbb{R}^{n}$, and terminal cost $\phi\left(\bm{\mu}_{1},\bm{\Sigma}_{1},\bm{\mu}_{d},\bm{\Sigma}_{d}\right)=\frac{1}{2}\|\bm{\mu}_1 - \bm{\mu}_{d}\|_2^2 + \frac{1}{2}\|\bm{\Sigma}_1-\bm{\Sigma}_{d}\|_{\mathrm{Frobenius}}^{2}$. Allowing \eqref{LinearController} to have an affine term: $\bm{u}_{t} = \bm{K}_{t}\bm{x}_t + \bm{v}_{t}$, we then define an extended version of the Hamiltonian \eqref{defHamiltonian} as the mapping $H:\mathcal{T}^{*}\mathbb{R}^{n}\times\mathcal{T}^{*}\mathbb{S}^{n}_{++}\times\mathbb{R}^{m\times n}\times\mathbb{R}^{n}\mapsto\mathbb{R}$, given by
\begin{align}
&H\left(\bm{\mu}_{t},\bm{z}_{t},\bm{\Sigma}_{t},\bm{P}_{t},\bm{K}_{t},\bm{v}_{t}\right) = \langle\bm{z}_{t},\left(\bm{A}_t + \bm{B}_t\bm{K}_t\right)\bm{\mu}_{t} + \bm{B}_{t}\bm{v}_{t}\rangle\nonumber\\
&+\langle\bm{\Sigma}_{t}+\bm{\mu}_{t}\bm{\mu}_{t}^{\top},\bm{K}_{t}^{\top}\bm{K}_{t} + \bm{Q}_{t}\rangle + 2\bm{v}_{t}^{\top}\bm{K}_{t}\bm{\mu}_{t} + \bm{v}_{t}^{\top}\bm{v}_{t}  \nonumber\\
&+ \langle\bm{P}_{t},  \left(\bm{A}_t + \bm{B}_t\bm{K}_t\right)\bm{\Sigma}_t+ \bm{\Sigma}_t\left(\bm{A}_t+ \bm{B}_t\bm{K}_t\right)^{\!\top} \!\!+ \bm{B}_t\bm{B}_t^{\top}\rangle.
\label{defModifiedHamiltonian}     
\end{align}
Computation similar to the proof of Proposition \ref{prop:FOOC} finds the optimal feedback gain $\bm{K}_{t}^{\mathrm{opt}}$ as in \eqref{OptimalGain}, and the optimal feedforward gain
\begin{align}
\bm{v}_{t}^{\mathrm{opt}}=\bm{B}_{t}^{\top}\left(\bm{P}_{t}^{\mathrm{opt}}\bm{\mu}_{t}^{\mathrm{opt}}-\bm{z}_{t}^{\mathrm{opt}}\right).
\label{vopt}    
\end{align}
The optimal mean and its costate pair $\left(\bm{\mu}_{t}^{\mathrm{opt}},\bm{z}_{t}^{\mathrm{opt}}\right)\in\mathbb{R}^{n}\times\mathbb{R}^{n}$ are solved from the vector boundary value problem 
\begin{subequations}
\begin{align}
&\begin{pmatrix}\dot{\bm{\mu}}_{t}^{\mathrm{opt}}\\
\dot{\bm{z}}_{t}^{\mathrm{opt}}
\end{pmatrix}
=\left[\begin{array}{cc}
\bm{A}_t & -\bm{B}_{t}\bm{B}_{t}^{\top} \\
-\bm{Q}_{t} & -\bm{A}_{t}^{\top}
\end{array}\right]\begin{pmatrix}\bm{\mu}_{t}^{\mathrm{opt}}\\
\bm{z}_{t}^{\mathrm{opt}}
\end{pmatrix}, \label{MeanCostateODE}\\
&\bm{\mu}_{0}\in\mathbb{R}^{n}\;\text{given},\quad\bm{z}_{1}^{\mathrm{opt}} = \frac{\partial\phi}{\partial\bm{\mu}_{1}^{\mathrm{opt}}} = \bm{\mu}_{1}^{\mathrm{opt}}-\bm{\mu}_{d}.
\end{align}
\label{MeanCostateTPBVP}    
\end{subequations}
In other words, \eqref{MeanCostateTPBVP} gets appended to \eqref{FOOC}. From \eqref{vopt}, $\bm{v}_{t}^{\mathrm{opt}}$ depends on $\bm{P}_{t}^{\mathrm{opt}}$, but $\bm{K}_{t}^{\mathrm{opt}}$ is independent of the solution of \eqref{MeanCostateTPBVP}. Thus, for given nonzero $\bm{\mu}_{0},\bm{\mu}_{d}$, the zero mean covariance controller can be synthesized independent of the feedforward control design. In the remaining of this work, we focus on the covariance steering.
\end{remark}

\begin{remark}\label{Remark:DirectNumericalSolOfFOOCchallenging}
Direct numerical solution of \eqref{FOOC} is problematic. At the ODE level, even though \eqref{CovarianceCostateODE} is decoupled, \eqref{CovarianceStateODE} is coupled. At the boundary condition level, even though \eqref{InitialCovariance} is decoupled, \eqref{TransveersalityCondition} is coupled. 
\end{remark}

To circumvent the computational difficulty mentioned above, we make use of the following change-of-variable proposed in \cite{chen2015optimal,ChenGeorgiouPavonPartIII}:
\begin{align}
\bm{H}_{t} := \bm{\Sigma}_{t}^{-1} - \bm{P}_{t} \qquad\forall t\in[t_0,t_1].
\label{defH}    
\end{align}
The works in \cite{chen2015optimal,ChenGeorgiouPavonPartIII} considered the case where the terminal state covariance $\bm{\Sigma}_{1}^{\mathrm{opt}}$ was prescribed as problem data. However, in our setting, neither $\bm{\Sigma}_{1}^{\mathrm{opt}}$ nor $\bm{P}_{1}^{\mathrm{opt}}$ is known. Only their linear combination is constrained by \eqref{TransveersalityCondition}. Because of this transversality constraint, it is not obvious what, if any, computational benefit may be reaped from \eqref{defH}.

Our idea now is to set up a system of equations that folds in \eqref{defH} with the ODEs for $\bm{P}_{t},\bm{H}_{t}\in\mathbb{S}^{n}$. For this purpose, the following result will be useful.
\begin{proposition}\label{prop:Hivp}
Let $\bm{\Sigma}_{t}\in\mathbb{S}^{n}_{++}$ solve \eqref{CovarianceStateODE} with initial condition \eqref{InitialCovariance}, and let $\bm{P}_{t}\in\mathbb{S}^{n}$ solve the Riccati matrix ODE \eqref{CovarianceCostateODE} with initial condition $\bm{P}_{0}\in\mathbb{S}^{n}$. Then $\bm{H}_{t}\in\mathbb{S}^{n}$ defined in \eqref{defH} solves the Riccati matrix ODE initial value problem
\begin{subequations}
\begin{align}
-\dot{\bm{H}}_{t} &= \bm{A}_{t}^{\top}\bm{H}_{t} + \bm{H}_{t}\bm{A}_{t} + \bm{H}_{t}\bm{B}_{t}\bm{B}_{t}^{\top}\bm{H}_{t} - \bm{Q}_{t}, \label{HmatrixRiccatiODE}\\
\bm{H}_{0} &= \bm{\Sigma}_{0}^{-1} - \bm{P}_{0}. \label{HmatrixRiccatiIC}
\end{align}
\label{HmatrixRiccatiODEivp} 
\end{subequations}
\end{proposition}
\begin{proof}
Taking the time derivative on both sides of \eqref{defH}, and then using \eqref{CovarianceStateODE}-\eqref{CovarianceCostateODE}, results in \eqref{HmatrixRiccatiODE}. Evaluating \eqref{defH} at $t=t_0$ gives \eqref{HmatrixRiccatiIC}.
\end{proof}
In light of Proposition \ref{prop:Hivp}, we are led to solve the coupled nonlinear system of equations:
\begin{subequations}
\begin{align}
\!\!-\dot{\bm{P}}_{t}^{\mathrm{opt}}\! &= \!\bm{A}_{t}^{\top}\bm{P}_{t}^{\mathrm{opt}} \!+\!\bm{P}_{t}^{\mathrm{opt}}\bm{A}_{t} - \!\bm{P}_t^{\mathrm{opt}}\bm{B}_{t}\bm{B}_{t}^{\top}\bm{P}_{t}^{\mathrm{opt}}\! + \!\bm{Q}_{t}, \label{CovarianceCostateODEFinal}\\
\!\!-\dot{\bm{H}}_{t} &= \bm{A}_{t}^{\top}\bm{H}_{t} + \bm{H}_{t}\bm{A}_{t} + \bm{H}_{t}\bm{B}_{t}\bm{B}_{t}^{\top}\bm{H}_{t} - \bm{Q}_{t}, \label{HmatrixRiccatiODEFinal}\\
\!\!\bm{H}_{0} &= \bm{\Sigma}_{0}^{-1} - \bm{P}_{0}, \label{HmatrixRiccatiICFinal}\\
\!\!\bm{H}_{1} &\stackrel{\eqref{TransveersalityCondition}, \eqref{defH}}{=} \left(\bm{P}_{1} + \bm{\Sigma}_{d}\right)^{-1} - \bm{P}_{1}, \label{HmatrixRiccatiTerminalCondFinal}
\end{align}
\label{CoupledSystemFinal}    
\end{subequations}
where $\bm{P}_{0}:=\bm{P}_{t=t_0}^{\mathrm{opt}}$, $\bm{P}_{1}:=\bm{P}_{t=t_1}^{\mathrm{opt}}$.

Since $\bm{\Sigma}_{0},\bm{\Sigma}_{d}\succ\bm{0}$ are given, and the solution maps for the Riccati ODEs \eqref{CovarianceCostateODEFinal}-\eqref{HmatrixRiccatiODEFinal} can be represented \cite[p. 156]{brockett2015finite}, \cite{shayman1986phase} as linear fractional transforms (LFTs), it is natural to speculate if \eqref{CoupledSystemFinal} can be solved via a recursive algorithm, thereby solving \eqref{FOOC}. In the following, we design such an algorithm. To do so, we need to construct mappings $$\bm{P}_{1}\mapsto\bm{P}_{0},\quad\bm{H}_{0}\mapsto\bm{H}_{1},\quad \bm{H}_{1}\mapsto\bm{P}_{1},$$
where the subscript zero corresponds to $t=t_0$, and the subscript one corresponds to $t=t_1$. These mappings are to be constructed in a way that ensures invertibility of $\bm{P}_{1}+\bm{\Sigma}_d$ in \eqref{HmatrixRiccatiTerminalCondFinal}; see Proposition \ref{prop:inertibilityofAumOfP1plusAndSigmad}.

In the next Sec. \ref{sec:Mappings}, we propose the constructions for these maps and detail their properties. The results from Sec. \ref{sec:Mappings} will be used for the design and analysis of the proposed algorithm in Sec. \ref{sec:RecrsiveAlgorithm}.

\begin{figure*}[th]
\centering
\includegraphics[width=0.75\linewidth]{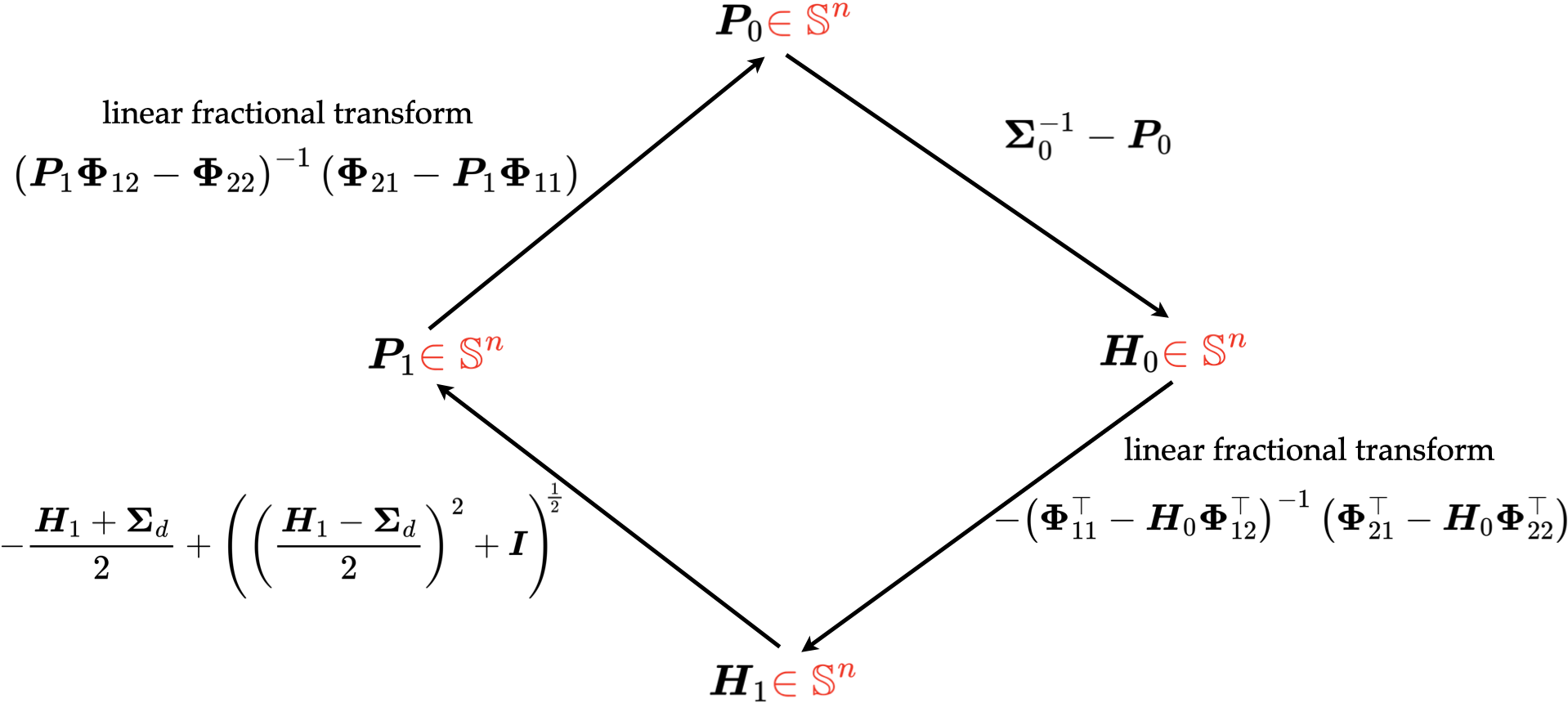}
\caption{The proposed fixed point recursion $\bm{P}_{0}\mapsto \left(\bm{P}_{0}\right)_{\text{next}}$ as a composition of four mappings.}
\label{fig:FixedPointRecursion}
\end{figure*}

\section{Construction of Matrix-valued Mappings}\label{sec:Mappings}

\subsection{Mappings \texorpdfstring{$\bm{P}_{1}\mapsto\bm{P}_{0},\bm{H}_{0}\mapsto\bm{H}_{1}$}{TEXT}}\label{subsec:P1mapstoP0andH0mapstoH1}
Let us denote the $2n\times 2n$ Hamiltonian matrix in \eqref{MeanCostateODE} as
\begin{align}
\bm{M}_{t} := \left[\begin{array}{cc}
\bm{A}_t & -\bm{B}_{t}\bm{B}_{t}^{\top} \\
-\bm{Q}_{t} & -\bm{A}_{t}^{\top}
\end{array}\right],
\label{DefHamiltonianMatrix}    
\end{align}
and its state transition matrix in block-partitioned form:
\begin{align}
\bm{\Phi}(s,t):= \left[\begin{array}{cc}
\bm{\Phi}_{11}(s,t) & \bm{\Phi}_{12}(s,t) \\
\bm{\Phi}_{21}(s,t) & \bm{\Phi}_{22}(s,t)
\end{array}\right], \quad t_{0}\leq s \leq t \leq t_{1},
\label{defSTMofM}    
\end{align}
where $\bm{\Phi}_{ij}(s,t)\in\mathbb{R}^{n\times n}$ $\forall (i,j)\in\{1,2\}^2$. By definition, 
\begin{align}
\partial_{t} \bm{\Phi}(s,t) = \bm{M}(t)\bm{\Phi}(s,t), \quad \bm{\Phi}(s,s)=\bm{I}.
\label{defSTM}    
\end{align} 

For convenience, consider slightly abused notation:
\begin{align}
\left[\begin{array}{cc}
\bm{\Phi}_{11} & \bm{\Phi}_{12}\\
\bm{\Phi}_{21} & \bm{\Phi}_{22}
\end{array}\right] := \left[\begin{array}{cc}
\bm{\Phi}_{11}(t_0,t_1) & \bm{\Phi}_{12}(t_0,t_1) \\
\bm{\Phi}_{21}(t_0,t_1) & \bm{\Phi}_{22}(t_0,t_1)
\end{array}\right].
\label{STMBlockFormAtBoundary} 
\end{align}
The following result from \cite{ChenGeorgiouPavonPartIII} will be useful.
\begin{lemma}\cite[Lemma 3]{ChenGeorgiouPavonPartIII}\label{Lemma:STMblockproperties}
The blocks of \eqref{STMBlockFormAtBoundary} satisfy the following identities:
\begin{subequations}
\begin{align}
& \bm{\Phi}_{11}^{\!\top} \bm{\Phi}_{22}-\bm{\Phi}_{21}^{\!\top}\bm{\Phi}_{12} =\bm{I}, \label{identity1}\\
& \bm{\Phi}_{12}^{\!\top}\bm{\Phi}_{22}-\bm{\Phi}_{22}^{\!\top}\bm{\Phi}_{12}=\bm{0}, \label{identity2}\\
& \bm{\Phi}_{21}^{\!\top}\bm{\Phi}_{11}-\bm{\Phi}_{11}^{\!\top}\bm{\Phi}_{21}=\bm{0},\label{identity3}\\
& \bm{\Phi}_{11} \bm{\Phi}_{22}^{\!\top}-\bm{\Phi}_{12} \bm{\Phi}_{21}^{\!\top}=\bm{I}, \label{identity4}\\
& \bm{\Phi}_{12}\bm{\Phi}_{11}^{\!\top}-\bm{\Phi}_{11}\bm{\Phi}_{12}^{\!\top}=\bm{0}, \label{identity5}\\
& \bm{\Phi}_{21}\bm{\Phi}_{22}^{\!\top}-\bm{\Phi}_{22}\bm{\Phi}_{21}^{\!\top}=\bm{0}.\label{identity6}
\end{align}
\label{STMblockidentities}
\end{subequations}
Furthermore, $\bm{\Phi}_{11}$ and $\bm{\Phi}_{12}$ are invertible.
\end{lemma}

The mapping $\bm{P}_{1}\mapsto\bm{P}_{0}$ is the solution of the Riccati matrix ODE initial value problem \eqref{CovarianceCostateODEFinal} backward in time from $t=t_1$ to $t=t_0$, with  $\bm{P}_{t}(t=t_1)=\bm{P}_{1}\in\mathbb{S}^{n}$ given. As is well-known \cite[p. 156]{brockett2015finite}, \cite{shayman1986phase}, this mapping is an LFT in terms of the blocks of \eqref{STMBlockFormAtBoundary}:
\begin{align}
\bm{P}_{0} = \left(\bm{P}_{1}\bm{\Phi}_{12} - \bm{\Phi}_{22}\right)^{-1}\left(\bm{\Phi}_{21}-\bm{P}_{1}\bm{\Phi}_{11}\right).
\label{P1toP0asLFT}    
\end{align}
Thus, \eqref{P1toP0asLFT} is the explicit form for the mapping $\bm{P}_{1}\mapsto\bm{P}_{0}$.

Likewise, the mapping $\bm{H}_{0}\mapsto\bm{H}_{1}$ is the solution of the Riccati matrix ODE initial value problem \eqref{HmatrixRiccatiODEFinal} forward in time from $t=t_0$ to $t=t_1$, with  $\bm{H}_{t}(t=t_0)=\bm{H}_{0}\in\mathbb{S}^{n}$ given. The counterpart of \eqref{P1toP0asLFT} is
\begin{align}
\bm{H}_{1} = -\left(\bm{\Phi}_{11}^{\top}-\bm{H}_{0}\bm{\Phi}_{12}^{\top}\right)^{-1}\left(\bm{\Phi}_{21}^{\top}-\bm{H}_{0}\bm{\Phi}_{22}^{\top}\right).
\label{H0toH1asLFT}    
\end{align}
Thus, \eqref{H0toH1asLFT} is the explicit form for the mapping $\bm{H}_{0}\mapsto\bm{H}_{1}$.

\begin{theorem}\label{Thm:LFTproperties}
The mappings $\bm{P}_{1}\mapsto\bm{P}_{0}$, $\bm{H}_{0}\mapsto\bm{H}_{1}$ given by \eqref{P1toP0asLFT} and \eqref{H0toH1asLFT} respectively, are bijections. Furthermore, they take symmetric matrices to symmetric matrices.    
\end{theorem}
\begin{proof}
\textbf{Bijections.} An LFT 
\begin{align}
x\mapsto (ax + b)(cx + d)^{-1},
\label{StandardLFT}
\end{align}
where $a,b,c,d$ are constant $n\times n$ matrices and $x$ is a variable $n\times n$ matrix, is bijective in its domain of definition if and only if \cite[p. 22]{potapov1988linear} the coefficient matrix 
$\begin{pmatrix}
a & b\\
c & d
\end{pmatrix}$
is nonsingular. The LFT \eqref{P1toP0asLFT} is a transposition of \eqref{StandardLFT} with coefficient matrix
\begin{align}
&\begin{bmatrix}
-\bm{\Phi}_{11}^{\top} & \bm{\Phi}_{21}^{\top}\\
\bm{\Phi}_{12}^{\top} & -\bm{\Phi}_{22}^{\top}
\end{bmatrix} =\begin{bmatrix}
-\bm{\Phi}_{11} & \bm{\Phi}_{12}\\
\bm{\Phi}_{21} & -\bm{\Phi}_{22}
\end{bmatrix}^{\top} \nonumber\\ =&\left(\begin{bmatrix}
-\bm{I} & \bm{0}\\
\bm{0} & \bm{I}
\end{bmatrix}\begin{bmatrix}
\bm{\Phi}_{11} & \bm{\Phi}_{12}\\
\bm{\Phi}_{21} & \bm{\Phi}_{22}
\end{bmatrix}\begin{bmatrix}
\bm{I} & \bm{0}\\
\bm{0} & -\bm{I}
\end{bmatrix}\right)^{\!\!\top},
\label{CoeffMatrixOfTransposedLFTP1}
\end{align}
which is nonsingular because the state transition matrix \eqref{STMBlockFormAtBoundary} is nonsingular. This proves that the LFT \eqref{P1toP0asLFT} is bijective.

Likewise, the LFT \eqref{H0toH1asLFT} is a transposition of the LFT \eqref{StandardLFT} with coefficient matrix
\begin{align}
&\begin{bmatrix}
-\bm{\Phi}_{22} & \bm{\Phi}_{21}\\
\bm{\Phi}_{12} & -\bm{\Phi}_{11}
\end{bmatrix} \!=\! \begin{bmatrix}
-\bm{I} & \bm{0}\\
\bm{0} & \bm{I}
\end{bmatrix}\!\begin{bmatrix}
\bm{\Phi}_{22} & \bm{\Phi}_{21}\\
\bm{\Phi}_{12} & \bm{\Phi}_{11}
\end{bmatrix}\!\begin{bmatrix}
\bm{I} & \bm{0}\\
\bm{0} & -\bm{I}
\end{bmatrix}\nonumber\\
=& \!\begin{bmatrix}
-\bm{I} & \bm{0}\\
\bm{0} & \bm{I}
\end{bmatrix}\!\begin{bmatrix}
\bm{0} & \bm{I}\\
\bm{I} & \bm{0}
\end{bmatrix}\!\begin{bmatrix}
\bm{\Phi}_{11} & \bm{\Phi}_{12}\\
\bm{\Phi}_{21} & \bm{\Phi}_{22}
\end{bmatrix}\!\begin{bmatrix}
\bm{0} & \bm{I}\\
\bm{I} & \bm{0}
\end{bmatrix}\!\begin{bmatrix}
\bm{I} & \bm{0}\\
\bm{0} & -\bm{I}
\end{bmatrix},
\label{CoeffMatrixOfTransposedLFTH0}
\end{align}
which is nonsingular because the state transition matrix \eqref{STMBlockFormAtBoundary} is nonsingular. Thus the LFT \eqref{H0toH1asLFT} is bijective.

\noindent\textbf{Symmetric to symmetric.} Let us define the $2n\times 2n$ skew-symmetric matrix $\bm{J} := \begin{bmatrix}
\bm{0} & \bm{I}\\
-\bm{I} & \bm{0}
\end{bmatrix}$. An LFT of the form \eqref{StandardLFT} takes symmetric matrices to symmetric matrices if and only if the following symplectic collinearity condition is satisfied \cite[Thm. 4]{potapov1988linear}: 
\begin{align}
\!\!\begin{pmatrix}
a & b\\
c & d
\end{pmatrix}^{\!\!\top}\!\!\bm{J}\begin{pmatrix}
a & b\\
c & d
\end{pmatrix}=r\bm{J}\quad\text{for some}\; r>0.
\label{collinearitycondition}    
\end{align}
Since the LFT \eqref{P1toP0asLFT} is a transposition of \eqref{StandardLFT} with coefficient matrix \eqref{CoeffMatrixOfTransposedLFTP1}, it suffices to demonstrate \eqref{collinearitycondition} for the coefficient matrix \eqref{CoeffMatrixOfTransposedLFTP1}. Direct computation verifies
\begin{align}
&\begin{bmatrix}
-\bm{\Phi}_{11} & \bm{\Phi}_{12}\\
\bm{\Phi}_{21} & -\bm{\Phi}_{22}
\end{bmatrix}^{\top}\!\!\bm{J}\begin{bmatrix}
-\bm{\Phi}_{11}^{\top} & \bm{\Phi}_{21}^{\top}\\
\bm{\Phi}_{12}^{\top} & -\bm{\Phi}_{22}^{\top}
\end{bmatrix} \nonumber\\
=&\begin{bmatrix}
\bm{\Phi}_{12}\bm{\Phi}_{11}^{\top} - \bm{\Phi}_{11}\bm{\Phi}_{12}^{\top} & -\bm{\Phi}_{12}\bm{\Phi}_{21}^{\top}+\bm{\Phi}_{11}\bm{\Phi}_{22}^{\top}\\
-\bm{\Phi}_{22}\bm{\Phi}_{11}^{\top} + \bm{\Phi}_{21}\bm{\Phi}_{12}^{\top} & \bm{\Phi}_{22}\bm{\Phi}_{21}^{\top} - \bm{\Phi}_{21}\bm{\Phi}_{22}^{\top}
\end{bmatrix}\nonumber\\
=&\begin{bmatrix}
\bm{0} & \bm{I}\\
-\bm{I} & \bm{0}
\end{bmatrix} = \bm{J},
\end{align}
where the last but one equality used \eqref{identity4}-\eqref{identity6} from Lemma \ref{Lemma:STMblockproperties}. Therefore, the condition \eqref{collinearitycondition} is satisfied with $r=1$, and the LFT \eqref{P1toP0asLFT} takes symmetric matrices to symmetric matrices. 

Similarly, the LFT \eqref{H0toH1asLFT} being a transposition of the LFT \eqref{StandardLFT} with coefficient matrix \eqref{CoeffMatrixOfTransposedLFTH0}, it suffices to demonstrate \eqref{collinearitycondition} for the coefficient matrix \eqref{CoeffMatrixOfTransposedLFTH0}. Again direct computation together with the identities \eqref{identity4}-\eqref{identity6} from Lemma \ref{Lemma:STMblockproperties} verify the same.
\end{proof}


\subsection{Mapping \texorpdfstring{$\bm{H}_1 \mapsto \bm{P}_1$}{TEXT}}\label{subsec:H1mapstoP1}
We start by noticing that \eqref{HmatrixRiccatiTerminalCondFinal}, or equivalently \eqref{TransveersalityCondition} and \eqref{defH} can be recast in the following two ways:
\begin{align}
\bm{I} &= \bm{\Sigma}_{1}\left(\bm{P}_1 + \bm{H}_1\right) = \left(\bm{P}_1 + \bm{\Sigma}_d\right)\left(\bm{P}_1 + \bm{H}_1\right),\label{FirstWay}\\
\bm{I} &= \left(\bm{P}_1 + \bm{H}_1\right)\bm{\Sigma}_1 = \left(\bm{P}_1 + \bm{H}_1\right)\left(\bm{P}_1 + \bm{\Sigma}_d\right), \label{SecondWay}
\end{align}
wherein $\bm{I}$ denotes the identity matrix.

The equations \eqref{FirstWay}-\eqref{SecondWay} can be equivalently written as
\begin{align}
\bm{P}_1^2 + \bm{P}_{1}\bm{H}_{1} + \bm{\Sigma}_d \bm{P}_{1} + \left(\bm{\Sigma}_d \bm{H}_1 - \bm{I}\right) &= \bm{0}, \label{FirstWayReduced}\\
\bm{P}_1^2 + \bm{P}_{1}\bm{\Sigma}_d +  \bm{H}_{1}\bm{P}_{1} + \left(\bm{H}_1 \bm{\Sigma}_d - \bm{I}\right) &= \bm{0}. \label{SecondWayReduced}
\end{align}
The structures of \eqref{FirstWayReduced} and \eqref{SecondWayReduced} reveal the following: if either \eqref{FirstWayReduced} or \eqref{SecondWayReduced} admits a \emph{symmetric} solution $\bm{P}_{1}$, that solution must satisfy both \eqref{FirstWayReduced} and \eqref{SecondWayReduced}. So in effect, we have one algebraic equation in one symmetric unknown $\bm{P}_1$.

Now the question becomes whether \eqref{FirstWayReduced} or \eqref{SecondWayReduced}, or some equivalent version of either, is more convenient for the mapping $\bm{H}_{1}\mapsto\bm{P}_{1}$. One particularly simple equivalent version is obtained by subtracting \eqref{SecondWayReduced} from \eqref{FirstWayReduced}, giving the Sylvester equation 
    \begin{align}
        (\bm{H}_1 - \bm{\Sigma}_d)\bm{P}_1 + \bm{P}_1(\bm{\Sigma}_d - \bm{H}_1) = \bm{\Sigma_}d\bm{H}_1 - \bm{H}_1\bm{\Sigma}_d.
    \label{SylvesterEq}    
    \end{align}
We deduce the following properties of its solution.
\begin{theorem}\label{thm:SylvesterSolution}
Given $\bm{H}_{1}$ symmetric and $\bm{\Sigma}_d \succ \bm{0}$, there exists non-unique $\bm{P}_{1}=\bm{P}_{1}^{\top}$ solving the Sylvester equation \eqref{SylvesterEq}.    
\end{theorem}
\begin{proof}
A symmetric solution to \eqref{SylvesterEq} exists because both sides of \eqref{SylvesterEq} have zero trace (necessary for existence), and if $\bm{P}_{1}$ solves \eqref{SylvesterEq} then so does $\bm{P}_{1}^{\top}$.    

However, \eqref{SylvesterEq} cannot admit unique solution because the coefficient matrices $\bm{H}_1-\bm{\Sigma}_d, \bm{\Sigma}_d-\bm{H}_1$ for the linear terms are negation of each other; see \cite{sylvester1884equation,bhatia1997and}.
\end{proof}
\begin{remark}
Particular symmetric solutions of \eqref{SylvesterEq} can be obtained by solving the system of linear equations $\left(\left(\bm{I}\otimes(\bm{H}_1 - \bm{\Sigma}_d)\right) - \left((\bm{H}_1 - \bm{\Sigma}_d)\otimes\bm{I}\right) \right){\mathrm{vec}}(\bm{P}_{1})={\mathrm{vec}}(\bm{\Sigma_}d\bm{H}_1 - \bm{H}_1\bm{\Sigma}_d)$. Two such solutions are $-\bm{H}_{1}$ and $-\bm{\Sigma}_{d}$, verified by direct substitution in \eqref{SylvesterEq}. In general, the dimension of the solution set depends on the number of distinct eigenvalues of $(\bm{H}_{1}-\bm{\Sigma}_{d})$. If all eigenvalues of $(\bm{H}_{1}-\bm{\Sigma}_{d})$ are distinct, then the solution set is of the form $\{-\bm{H}_1+ c_0 \bm{I} + c_1 (\bm{H}_{1}-\bm{\Sigma_{d}}) + \cdots + c_{n-1} (\bm{H}_{1}-\bm{\Sigma_{d}})^{n-1}\ | \ c_0, c_{1}, \hdots, c_{n-1}\in\mathbb{R}\}$, which includes $-\bm{\Sigma}_{d}$. This is a consequence of a result of Taussky and Zassenhaus \cite{taussky1959similarity}.
\end{remark}

For given $\bm{H}_{1}\in\mathbb{S}^{n},\bm{\Sigma}_{d}\succ\bm{0}$, it is unclear which of the non-unique symmetric solutions of \eqref{SylvesterEq} can be used for the mapping $\bm{H}_{1}\mapsto\bm{P}_{1}$, that can in turn guarantee convergence of the recursive algorithm. This analysis suggests that instead of removing the nonlinearity from \eqref{FirstWayReduced}-\eqref{SecondWayReduced}, we should try to leverage the same.

As an alternative to \eqref{SylvesterEq}, adding \eqref{FirstWayReduced} and \eqref{SecondWayReduced} gives an instance of Continuous-time Algebraic Riccati Equation (CARE): 
    \begin{align}
        \bm{P}_1^2 + \bm{P}_1\left(\frac{\bm{H}_1 + \bm{\Sigma}_d}{2} \right) + \left(\frac{\bm{H}_1 + \bm{\Sigma}_d}{2} \right)\bm{P}_1 \notag \\
        + \left(\frac{\bm{\Sigma}_d \bm{H}_1 + \bm{H}_1\bm{\Sigma}_d}{2} - \bm{I}\right) = \bm{0}.
    \label{SpecificCARE}    
    \end{align}
    The next result guarantees the existence and non-uniqueness for the symmetric solution $\bm{P}_{1}$ of \eqref{SpecificCARE}. It also finds the unique symmetric solution that is stabilizing.
\begin{theorem}\label{Thm:CARE}
Given $\bm{H}_{1}$ symmetric and $\bm{\Sigma}_d \succ \bm{0}$, there exists non-unique $\bm{P}_{1}=\bm{P}_{1}^{\top}$ solving the CARE \eqref{SpecificCARE}. Among these non-unique symmetric solutions, the only one that is stabilizing, is given by
\begin{align}
\bm{P}_{1} = -\frac{\bm{H}_{1}+\bm{\Sigma}_{d}}{2} + \left(\left(\frac{\bm{H}_{1}-\bm{\Sigma}_{d}}{2}\right)^{\!\!2} + \bm{I}\right)^{\frac{1}{2}}.
\label{P1explicit}    
\end{align}
\end{theorem}
\begin{proof}
\textbf{Existence of symmetric solution.} We rewrite \eqref{SpecificCARE} in the standard CARE form:
\begin{align}
\bm{P}_{1}\bar{\bm{B}}\bm{P}_{1} + \bm{P}_{1}\bar{\bm{A}}+\bar{\bm{A}}^{\top}\bm{P}_{1} - \bar{\bm{C}} = \bm{0},
\label{StandardCARE}    
\end{align}
with $\bar{\bm{A}} = \frac{\bm{H}_1 + \bm{\Sigma}_d}{2}$, $\bar{\bm{B}} = \bm{I}$, $\bar{\bm{C}} =\bm{I}-\frac{1}{2}\left(\bm{\Sigma}_d \bm{H}_1 + \bm{H}_1\bm{\Sigma}_d\right)$.
The LTI pair
$$\left(\bar{\bm{A}},\bar{\bm{B}}\right) := \left(\frac{\bm{H}_1 + \bm{\Sigma}_d}{2},\bm{I}\right)$$
has the associated controllability Gramian
\begin{align}
\int_{s}^{t}e^{\bar{\bm{A}}\tau}\bar{\bm{B}}\bar{\bm{B}}^{\top}e^{\bar{\bm{A}}^{\top}\tau}\differential\tau = \int_{s}^{t}e^{\left(\bm{H}_1 + \bm{\Sigma}_{d}\right)\tau}\differential\tau
\label{CtrbGramian}    
\end{align}
for any $[s,t]\subseteq[t_0,t_1]$. Since $\left(\bm{H}_1 + \bm{\Sigma}_{d}\right)$ is symmetric, the matrix $e^{\left(\bm{H}_1 + \bm{\Sigma}_{d}\right)\tau}\succ\bm{0}$ for all $\tau\in[s,t]$. Therefore, the Gramian \eqref{CtrbGramian} is positive definite for any $[s,t]\subseteq[t_0,t_1]$. Hence the pair $(\bar{\bm{A}},\bar{\bm{B}})$ is controllable. 

It is known \cite[Theorem 1]{lancaster1980existence} that for $(\bar{\bm{A}},\bar{\bm{B}})$ controllable, $\bar{\bm{B}}\succeq 0$ and $\bar{\bm{C}}$ symmetric, the \emph{existence of symmetric solution} $\bm{P}_1=\bm{P}_1^{\top}$ for the CARE \eqref{StandardCARE} 
is guaranteed. Hence, there exists a symmetric solution for \eqref{SpecificCARE}.

\noindent\textbf{Non-uniqueness of the symmetric solution.} Recall that the symmetric solution $\bm{P}_{1}$ for \eqref{StandardCARE} is \emph{unique if and only if} \cite[Corollary 4]{lancaster1980existence}, \cite[Lemma 3.2.3]{kuvcera1991algebraic}
\begin{align}
\spec\left(\bm{N}:= \begin{bmatrix}
\bar{\bm{A}} & \bar{\bm{B}}\\
\bar{\bm{C}} & -\bar{\bm{A}}^{\top}
\end{bmatrix}\right)
\label{defM}    
\end{align}
is pure imaginary with no odd partial multiplicities. The partial multiplicities of an eigenvalue of a matrix are defined as the sizes of its corresponding Jordan blocks in the Jordan normal form of that matrix.

For the nonsingular $2n\times 2n$ matrix $\bm{J} := \begin{bmatrix}
\bm{0} & \bm{I}\\
-\bm{I} & \bm{0}
\end{bmatrix}$,
and $\bm{N}$ as in \eqref{defM}, we directly verify that
\begin{align}
\bm{N} = \bm{J} \bm{N}^{\top} \bm{J} &= \bm{J}\left(-\bm{N}^{\top}\right)\left(-\bm{J}\right) \nonumber\\
&= \left(-\bm{J}\right)^{-1}\left(-\bm{N}^{\top}\right)\left(-\bm{J}\right),
\label{MSimilarityTransform}
\end{align}
where the last equality uses $\left(-\bm{J}\right)^{-1} = \bm{J}$. By \eqref{MSimilarityTransform}, $\bm{N}$ is similar to $-\bm{N}^{\top}$. Therefore, $\spec(\bm{N})$ has reflective symmetry about the pure imaginary axis. Note that this uses the symmetries of $\bar{\bm{B}}, \bar{\bm{C}}$ but not that of $\bar{\bm{A}}$. However, our $\bar{\bm{A}} = \frac{\bm{H}_1 + \bm{\Sigma}_d}{2}$ is additionally symmetric, so the spectrum of our $\bm{N}$ must be real of the form $\{\pm\sigma_i\}_{i=1}^{n}$ where $\sigma_{i}\in\mathbb{R}\forall i\in[n]$. Thus, a symmetric $\bm{P}_1$ solving \eqref{SpecificCARE} exists but is non-unique.

\noindent\textbf{Unique stabilizing symmetric solution.} We showed earlier that $(\bar{\bm{A}},\bar{\bm{B}})$ is a controllable pair. By a classic result due to Willems \cite[Thm. 5]{1099831}, for $(\bar{\bm{A}},\bar{\bm{B}})$ controllable, there is a unique $\bm{P}_{1}=\bm{P}_{1}^{\top}$ that is stabilizing. 

On the other hand, completion-of-squares for \eqref{SpecificCARE} gives
\begin{align}
\left(\bm{P}_{1}+\frac{\bm{H}_1 + \bm{\Sigma}_d}{2}\right)^{\!\!2} = \left(\frac{\bm{H}_{1}-\bm{\Sigma}_{d}}{2}\right)^{\!\!2} + \bm{I}.
\label{CompletionOfSquares}
\end{align}
Since $(\bm{H}_{1}-\bm{\Sigma}_{d})/2$ is symmetric, its square is positive semidefinite. So the right-hand-side of \eqref{CompletionOfSquares} is positive definite, and has unique positive definite (principal) square root. Denoting this principal square root as $(\cdot)^{\frac{1}{2}}$, then \eqref{CompletionOfSquares} yields two symmetric solutions
\begin{align}
\bm{P}_{1}^{\pm} = -\frac{\bm{H}_{1}+\bm{\Sigma}_{d}}{2} \pm \left(\left(\frac{\bm{H}_{1}-\bm{\Sigma}_{d}}{2}\right)^{\!\!2} + \bm{I}\right)^{\frac{1}{2}},
\label{P1plusminus}
\end{align}
and the corresponding closed-loop matrices are $$\bar{\bm{A}}^{\pm}:=-\bar{\bm{A}}-\bar{\bm{B}}\bar{\bm{B}}^{\top}\bm{P}_{1}^{\pm} = \mp \left(\left(\frac{\bm{H}_{1}-\bm{\Sigma}_{d}}{2}\right)^{\!\!2} + \bm{I}\right)^{\frac{1}{2}}.$$ As $\bar{\bm{A}}^{+}\preceq\bm{0}$ and $\bar{\bm{A}}^{-}\succeq\bm{0}$, we conclude $\bm{P}_{1}^{+}$ is stabilizing and $\bm{P}_{1}^{-}$ is anti-stabilizing. Because we already argued the existence of unique stabilizing solution, $\bm{P}_{1}^{+}$ in \eqref{P1explicit} must be that one.
\end{proof}

\begin{remark}
It is known \cite[Thm. 5]{1099831} that if $\bm{P}_{1}$ is any of the non-unique symmetric solutions of \eqref{SpecificCARE}, then it satisfies the following Loewner ordering: $\bm{P}_{1}^{-}\preceq\bm{P}_{1}\preceq\bm{P}_{1}^{+}$ where $\bm{P}_{1}^{\pm}$ are given by \eqref{P1plusminus}. Note that for given $\bm{H}_{1}\in\mathbb{S}^{n}$ and $\bm{\Sigma}_{d}\in\mathbb{S}^{n}_{++}$, the matrices $\bm{P}^{\pm}\in\mathbb{S}^{n}$ in \eqref{P1plusminus} are sign-indefinite in general.
\end{remark}

Theorem \ref{Thm:CARE} shows that requiring stabilizability is a way to extract \emph{unique} symmetric solution of the CARE \eqref{SpecificCARE}. In Sec. \ref{sec:RecrsiveAlgorithm}, we will use this stabilizing symmetric solution \eqref{P1explicit} for the mapping $\bm{H}_{1}\mapsto\bm{P}_{1}$, i.e., we set 
\begin{align}
\bm{P}_{1}\equiv\bm{P}_{1}^{+}.
\label{P1equalsP1plus}    
\end{align}
We will prove that this choice guarantees convergence of the recursive algorithm proposed in Sec. \ref{sec:RecrsiveAlgorithm}. Before doing so, let us point out that the choice \eqref{P1equalsP1plus} ensures the invertibility of $\bm{P}_{1}+\bm{\Sigma}_d$ in \eqref{HmatrixRiccatiTerminalCondFinal}.

\begin{proposition}\label{prop:inertibilityofAumOfP1plusAndSigmad}
Given $\bm{\Sigma}_{d}\succ\bm{0}$ and the choice \eqref{P1equalsP1plus} for the symmetric matrix $\bm{P}_{1}\in\mathbb{S}^{n}$, the invertibility of $\bm{P}_{1}+\bm{\Sigma}_d$ is guaranteed, i.e., the map \eqref{HmatrixRiccatiTerminalCondFinal} is well-defined.   
\end{proposition}
\begin{proof}
Let $\bm{Z}:=\dfrac{\bm{\Sigma}_d - \bm{H}_{1}}{2}\in\mathbb{S}^{n}$, $\bm{Z}_{1}:=\left(\bm{Z}^2 + \bm{I}\right)^{\frac{1}{2}}\in\mathbb{S}^{n}_{++}$. With the choice \eqref{P1equalsP1plus}, we have
\begin{align}
\bm{P}_{1}+\bm{\Sigma}_{d} = \bm{Z} + \bm{Z}_{1}.
\label{ZplusZ1}    
\end{align}
Write the spectral decomposition $\bm{Z}=\bm{V}\bm{\Lambda}\bm{V}^{\top}$ for $\bm{V}$ orthogonal, and $\bm{\Lambda}={\mathrm{diag}}(\lambda_i(\bm{Z}))$. Then $\bm{Z}_{1} = \bm{V}\left(\bm{\Lambda}^{2} + \bm{I}\right)^{\frac{1}{2}}\bm{V}^{\top}$. Thus $\bm{Z}$ and $\bm{Z}_{1}$ commute, and are simultaneously diagnoalizable. Therefore, the eigenvalues of \eqref{ZplusZ1} are $\lambda_i(\bm{Z}) + \sqrt{\left(\lambda_i(\bm{Z})\right)^2 + 1}$. Since $\sqrt{\lambda_i^2 + 1}>\vert\lambda_i\vert$ $\forall \lambda_i\in\mathbb{R}$, each $\lambda_i + \sqrt{\lambda_i^2 + 1} > 0$. So $\bm{P}_{1}+\bm{\Sigma}_{d} \succ\bm{0}$, and is invertible. 
\end{proof}


\begin{figure*}
    \centering
\subfloat[]{%
  \includegraphics[width=0.32\linewidth]{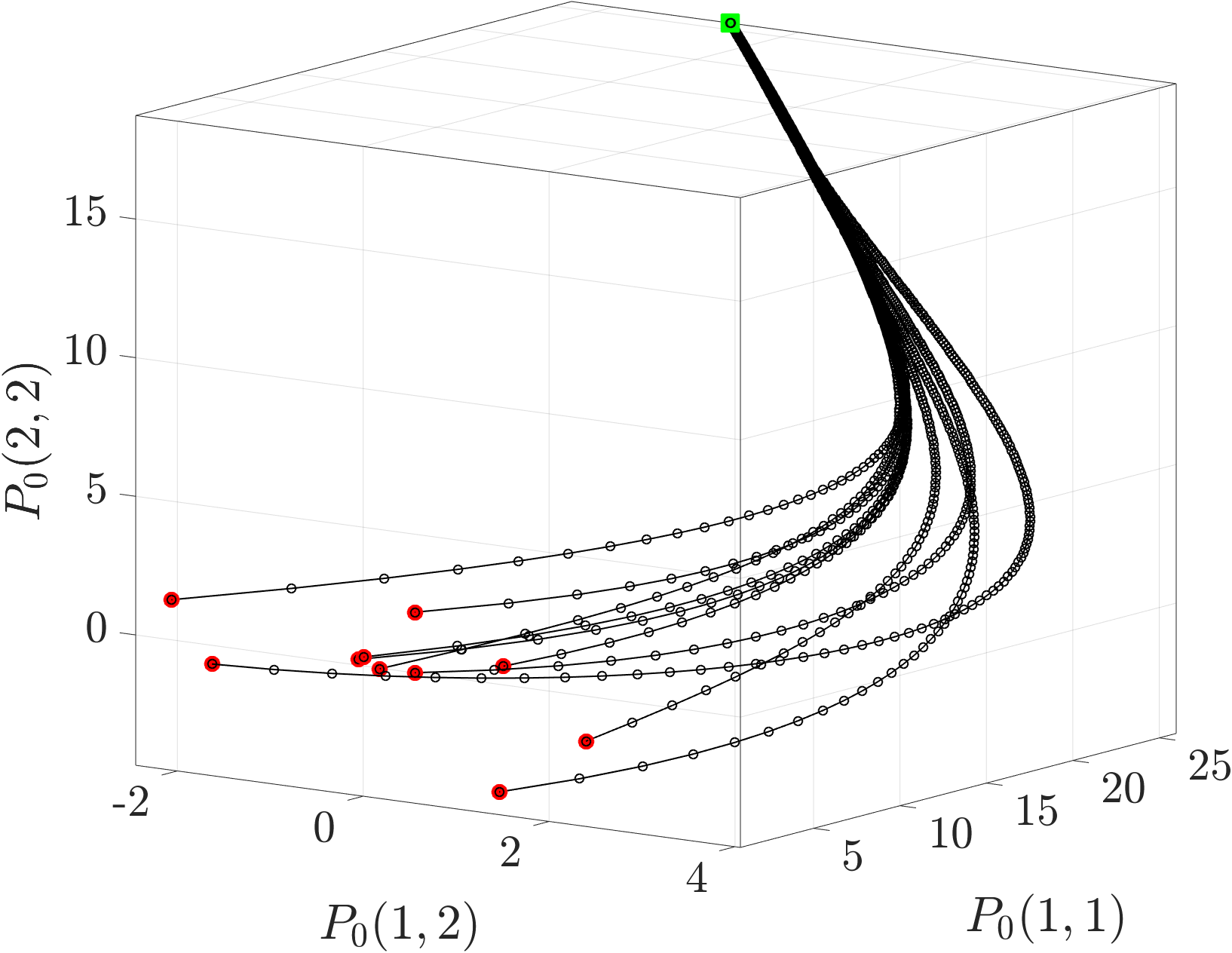}%
  \label{fig:recursionconvergenceS2}%
}\quad
\subfloat[]{%
  \includegraphics[width=0.32\linewidth]{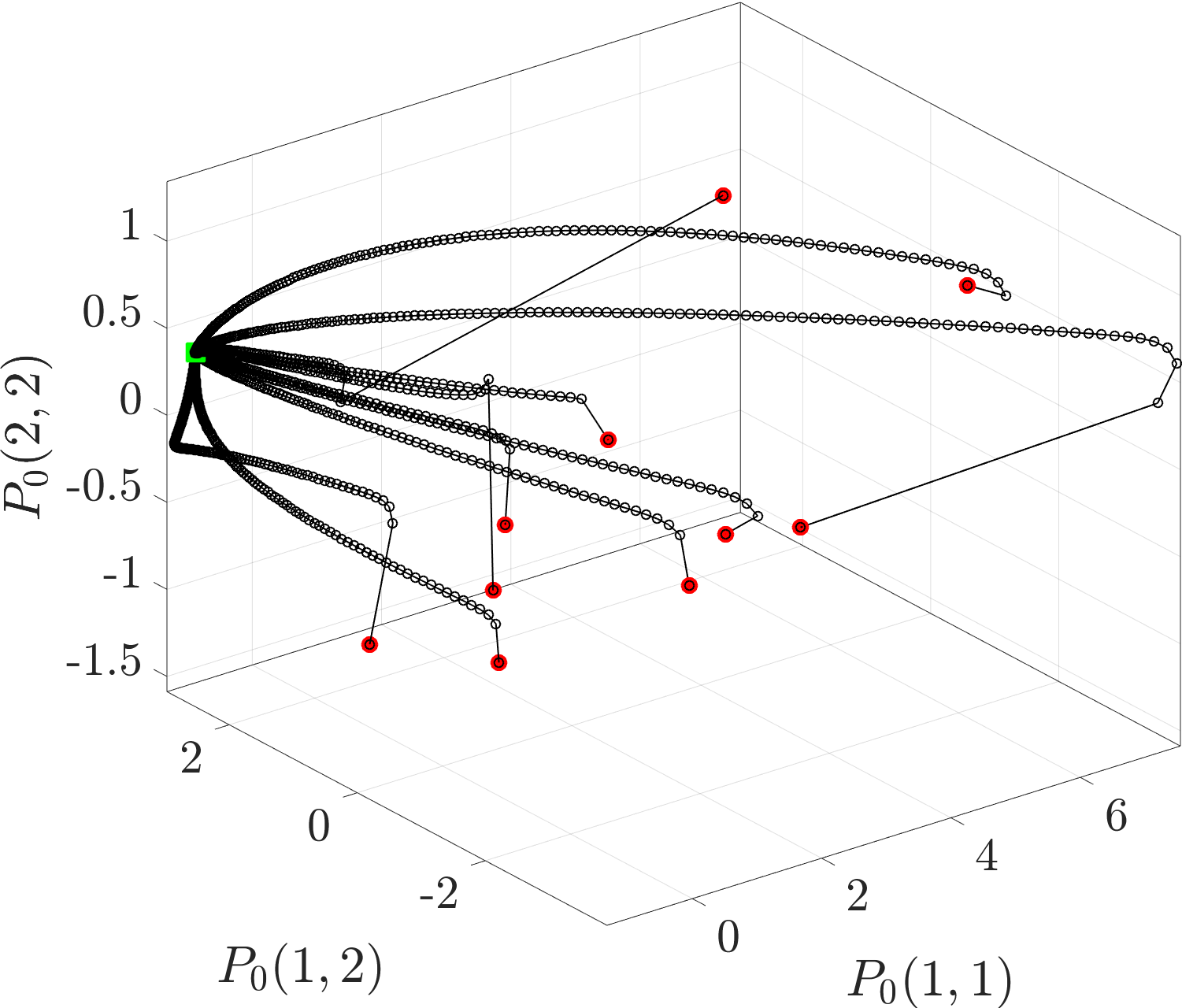}%
  \label{fig:recursionconvergenceS2another}%
}\quad
\subfloat[]{%
  \includegraphics[width=0.32\linewidth]{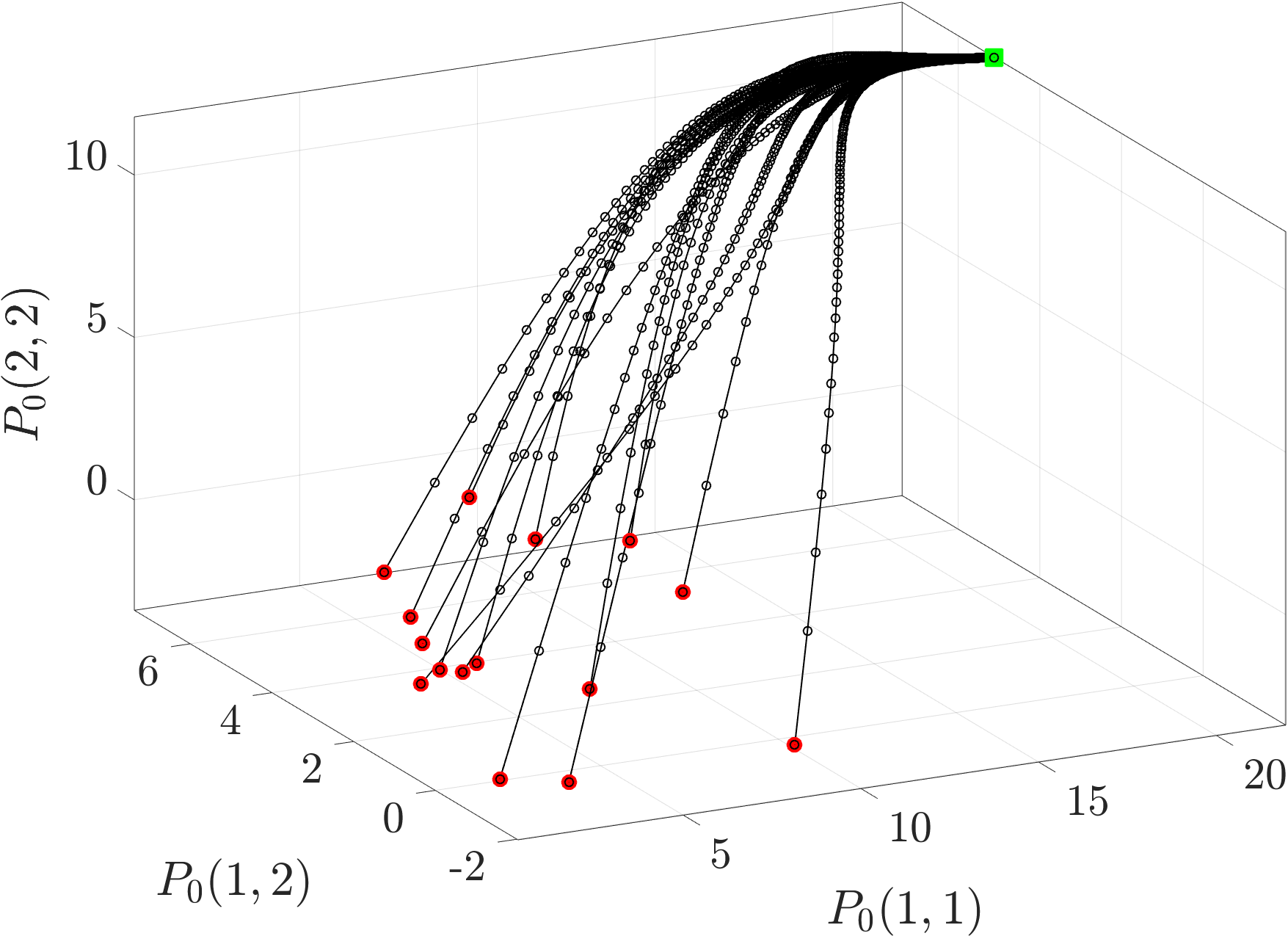}%
  \label{fig:recursionconvergenceS2yetanother}%
}
\caption{The phase portrait of the recursion $\bm{P}_{0}\mapsto \left(\bm{P}_{0}\right)_{\text{next}}$ proposed in Sec. \ref{sec:RecrsiveAlgorithm}. The three subplots (a)-(c) are for three randomly generated problem data $(\bm{A},\bm{B},\bm{Q},\bm{\Sigma}_0,\bm{\Sigma}_{d})$ with controllable $(\bm{A},\bm{B})\in\mathbb{R}^{2\times 2}\times\mathbb{R}^{2\times 1}$, $\bm{Q}\in\mathbb{S}^{2}_{+}$, and $\bm{\Sigma}_0,\bm{\Sigma}_{d}\in\mathbb{S}^{2}_{++}$. For fixed problem data, each subplot shows the convergence of the proposed recursion for random initial guesses $\bm{P}_{0}\in\mathbb{S}^{2}$ highlighted as red circular markers. In each subplot, the black circular markers show $1000$ iterates, and the green square is the converged $\bm{P}_{0}\in\mathbb{S}^{2}$.}
\label{fig:P0RecursionPhasePortrait}
\end{figure*}

\section{Recursive Algorithm}\label{sec:RecrsiveAlgorithm}
Building on the results from Sec. \ref{sec:MainIdeas} and \ref{sec:Mappings}, we propose a fixed point recursion of the form $\bm{P}_{0}\mapsto \left(\bm{P}_{0}\right)_{\text{next}}$ to solve \eqref{FOOC}. The proposed recursion, shown in Fig. \ref{fig:FixedPointRecursion}, is a composition of four mappings:
\begin{itemize}
\item $\bm{P}_{0}\to \bm{H}_{0}$ by evaluating \eqref{HmatrixRiccatiICFinal},

\item $\bm{H}_{0}\to \bm{H}_{1}$ by evaluating \eqref{H0toH1asLFT},

\item $\bm{H}_{1}\to \bm{P}_{1}$ by evaluating \eqref{P1explicit},

\item $\bm{P}_{1}\to \left(\bm{P}_{0}\right)_{\text{next}}$ by evaluating \eqref{P1toP0asLFT}.
\end{itemize}
Notice that all four mappings above, and thus their composition, are guaranteed to map $\mathbb{S}^{n}\to\mathbb{S}^{n}$. So it suffices to perform the proposed recursion over the half-vectorization of $\bm{P}_{0}$, i.e., over $n(n+1)/2$ reals instead of $n^2$ reals. This is computationally helpful for large $n$.

We next discuss the convergence of the proposed recursive algorithm.

\subsection{Convergence}
The proposed algorithm is a nonlinear recursion over the space of symmetric matrices. This recursion of the form $\bm{X}\mapsto\bm{F}(\bm{X})$, $\bm{X}\in\mathbb{S}^{n}$, where the matricial map $\bm{F}:\mathbb{S}^{n}\to\mathbb{S}^{n}$.
Our main result is Theorem \ref{thm:mainresult} that builds on auxiliary lemmas \ref{Lemma:F4bound}-\ref{Lemma:NondegenWideBlocks} in Appendix \ref{App:LemmaF4bound}. Specifically, Theorem \ref{thm:mainresult} establishes the existence-uniqueness of the fixed point $\bm{X}^{\star}\in\mathbb{S}^{n}$ of the proposed recursion, and its convergence. 

A natural proof strategy for such a result would be to show that the proposed recursion is contractive over the affine set $\mathbb{S}^{n}$ w.r.t. some metric such as the Frobenius or spectral norm of the difference. However, numerical experiments confirm that such contractive property does not hold for the proposed recursion $\bm{X}\mapsto\bm{F}(\bm{X})$. 

Another strategy could be separately arguing the existence-uniqueness of fixed point, and then establish the convergence by considering the function $V:\mathbb{S}^{n}\mapsto\mathbb{R}_{\geq 0}$ given by $V(\bm{X}):=\|\bm{X}-\bm{X}^{\star}\|_{\mathrm{Frobenius}}^{2}=\tr\left(\bm{X}-\bm{X}^{\star}\right)^{2}$ as a candidate Lyapunov function. However, numerics show that $V(\bm{F}(\bm{X})) \nless V(\bm{X})$ for all $\bm{X}\in\mathbb{S}^{n}\setminus\{\bm{X}^{\star}\}$. It is also unclear what could be an alternative choice of candidate Lyapunov function. 

Last but not the least, it is not apparent how to handle possible non-existence of inverses appearing in the LFTs \eqref{P1toP0asLFT}-\eqref{H0toH1asLFT}. Due to these difficulties, the proposed recursion calls for a careful reasoning.

\begin{theorem}\label{thm:mainresult}
Given the state transition matrix \eqref{defSTMofM} associated with \eqref{DefHamiltonianMatrix}, and $\bm{\Sigma}_0, \bm{\Sigma}_{d}\in\mathbb{S}^{n}_{++}$, define the mappings
\begin{subequations}
\begin{align}
\bm{F}_{1}\left(\bm{X}\right) &:= \bm{\Sigma}_{0}^{-1} - \bm{X}, \label{defF1}\\
\bm{F}_{2}\left(\bm{X}\right) &:= -\left(\bm{\Phi}_{11}^{\top}-\bm{X}\bm{\Phi}_{12}^{\top}\right)^{-1}\left(\bm{\Phi}_{21}^{\top}-\bm{X}\bm{\Phi}_{22}^{\top}\right), \label{defF2}\\
\bm{F}_{3}\left(\bm{X}\right) &:= -\frac{\bm{X}+\bm{\Sigma}_{d}}{2} + \left(\left(\frac{\bm{X}-\bm{\Sigma}_{d}}{2}\right)^{\!\!2} + \bm{I}\right)^{\frac{1}{2}}, \label{defF3}\\
\bm{F}_{4}\left(\bm{X}\right) &:= \left(\bm{X}\bm{\Phi}_{12} - \bm{\Phi}_{22}\right)^{-1}\left(\bm{\Phi}_{21}-\bm{X}\bm{\Phi}_{11}\right), \label{defF4}
\end{align}
\label{F1F2F3F4}    
\end{subequations}
in variable $\bm{X}\in\mathbb{S}^{n}$. The square root in \eqref{defF3} is understood as the principal square root. Then
\begin{itemize}
\item the composite endomorphism $\bm{F}:=\bm{F}_{4}\circ\bm{F}_{3}\circ\bm{F}_{2}\circ\bm{F}_{1}$
that maps $\mathbb{S}^{n}\to\mathbb{S}^{n}$, has a unique fixed point,

\item the recursion $\bm{P}_{0}\mapsto \left(\bm{P}_{0}\right)_{\text{next}}=\bm{F}\left(\bm{P}_{0}\right)$ converges to this unique fixed point for almost every $\bm{P}_{0}\in\mathbb{S}^{n}$.
\end{itemize}
\end{theorem}
\begin{proof}
\textbf{Existence-uniqueness of fixed point.} From Theorem \ref{Thm:LFTproperties}, the LFTs \eqref{defF2}-\eqref{defF4}, in their domain of definitions, are bijective. So is the linear map \eqref{defF1}. Since $(\bm{X}-\bm{\Sigma}_{d})/2$ is symmetric, its square is positive semidefinite (by spectral theorem), so the square root in \eqref{defF3} acts on a positive definite matrix. Because principal square root is bijective over the positive semidefinite cone, the mapping \eqref{defF3} is also bijective. As a result, the composition $\bm{F}:=\bm{F}_{4}\circ\bm{F}_{3}\circ\bm{F}_{2}\circ\bm{F}_{1}$ is bijective. Hence $\bm{F}$ can have \emph{at most one} fixed point.

Now we argue that $\bm{F}$ has \emph{at least one} fixed point. To do so, we note that the set $\mathbb{S}^{n}$ is closed and convex (w.r.t. standard Euclidean norm topology). The composite map $\bm{F}:\mathbb{S}^{n}\to\mathbb{S}^{n}$ is continuous and bounded. This can be seen from the fact that by Lemma \ref{Lemma:F4bound} in Appendix \ref{App:LemmaF4bound}, the map $\bm{F}_{4}$ is bounded, and by similar argument, $\bm{F}_{2}$ is bounded as well. The map $\bm{F}_{1}$, being affine, is non-expansive. By Lemma \ref{Lemma:F3isNonexpansive} in Appendix \ref{App:LemmaF4bound}, the map $\bm{F}_{3}$ is non-expansive as well. Finally, the $\bm{F}_{i}$ are continuous for all $i=1,\hdots,4$. So for any $\bm{X}\in\mathbb{S}^{n}$ with finite norm, the half-vectorization map $\bm{f}:={\mathrm{vech}}(\bm{F}(\bm{X})):\mathbb{R}^{n(n+1)/2}\to \mathbb{R}^{n(n+1)/2}$ is uniformly bounded by some finite $M>0$, i.e.,
\begin{align}
\underset{i\in\{1,\hdots,\frac{n(n+1)}{2}\}}{\sup}\quad\underset{\bm{x}\in\mathbb{R}^{\frac{n(n+1)}{2}}}{\sup}\vert f_{i}(\bm{x})\vert\leq M.
\label{UniformBoundVectorizedMap}    
\end{align}
Hence $\bm{f}$ maps the compact convex set $[-M,M]^{n(n+1)/2}$ to itself. By Brouwer's fixed point theorem \cite[p. 14]{milnor1965topology}, $\bm{f}$ and thus $\bm{F}$, admits a fixed point. 

Combining the ``at most one" and ``at least one" arguments, we conclude that the map $\bm{F}$ has a unique fixed point.

\noindent\textbf{Convergence.} 
To set aside the issue of possible initial conditions $\bm{X}^{(0)}\in\mathbb{S}^{n}$ for which the inverses in the LFTs \eqref{defF2}, \eqref{defF4} may not exist during the recursion, consider for now only those $\bm{X}^{(0)}\in\mathbb{S}^{n}$ for which $\{\bm{X}^{(k)}\}_{k\in\mathbb{N}}$ remains a well-defined sequence of real symmetric matrices generated by $\bm{X}^{(k)}=\bm{F}^{(k)}(\bm{X}^{(0)})$, where $\bm{F}^{(k)}$ denotes $k$-fold composition of $\bm{F}$. Call such initial conditions as ``good" $\bm{X}^{(0)}\in\mathbb{S}^{n}$. We will address its complement, i.e., ``bad" $\bm{X}^{0}\in\mathbb{S}^{n}$ leading to LFT singularity in a bit.

By \eqref{UniformBoundVectorizedMap}, the sequence $\{\bm{X}^{(k)}\}_{k\in\mathbb{N}}$ generated by ``good" $\bm{X}^{0}$, is bounded, and therefore, by the Bolzano-Weierstrass theorem \cite[Thm. 3.6(b)]{rudin1976principles}, has a convergent subsequence. Since $\mathbb{S}^{n}$ is closed, the limit point of the
convergent subsequence must also be in $\mathbb{S}^{n}$. Furthermore, by continuity of $\bm{F}$, this limit point must be a fixed point of
the map $\bm{X}\mapsto\bm{F}(\bm{X})$. 

On the other hand, we already proved the existence-uniqueness of fixed point for $\bm{F}:\mathbb{S}^{n}\to\mathbb{S}^{n}$. So the full sequence $\{\bm{X}^{(k)}\}$ converges to the same fixed point (since otherwise, we could extract a subsequence bounded away from this fixed point, which would converge to a different limit point, contradicting uniqueness). This proves that convergence to unique fixed point is guaranteed for all ``good" $\bm{X}^{(0)}\in\mathbb{S}^{n}$.

Now let us address a ``bad" $\bm{X}^{(0)}\in\mathbb{S}^{n}$ that leads to singularity in any of the two LFTs \eqref{defF2}, \eqref{defF4}. Notice that both these LFTs are of the form 
\begin{align}
(\bm{R} - \bm{X}\bm{S})^{-1}\left(\bm{X}\tilde{\bm{R}} - \tilde{\bm{S}}\right),
\label{F2F4genericform}    
\end{align}
with $\bm{R}=\bm{\Phi}_{11}^{\top}$, $\bm{S}=\bm{\Phi}_{11}^{\top}$ for \eqref{defF2}, and $\bm{R}=\bm{\Phi}_{22}$, $\bm{S}=\bm{\Phi}_{12}$ for \eqref{defF4}.
From Lemma \ref{Lemma:NondegenWideBlocks}, the wide matrix $\left[\bm{R},-\bm{S}\right]$ is nondegenerate (Definition \ref{def:NondegerateMatrix}) for both \eqref{defF2} and \eqref{defF4}. By \cite[Thm. 1]{potapov1988linear}, for $\left[\bm{R},-\bm{S}\right]$ nondegenerate, the collection of $\bm{X}\in\mathbb{S}^{n}$ such that $\bm{R} - \bm{X}\bm{S}$ is nonsingular is dense in $\mathbb{S}^{n}$. Thus, the collection of ``bad" $\bm{X}^{(0)}\in\mathbb{S}^{n}$ is a set of measure zero, i.e., convergence to unique fixed point is guaranteed for almost every $\bm{X}^{(0)}\in\mathbb{S}^{n}$.
\end{proof}

Fig. \ref{fig:P0RecursionPhasePortrait} illustrates the convergence of the proposed algorithm for $n=2$. Our numerical experiments presented in Sec. \ref{sec:examples} next, find that the convergence is fast in practice.

\begin{algorithm}[t]
\caption{Solve LQ covariance steering problem \eqref{OCP} with terminal cost \eqref{DefTerminalCost}}\label{alg:Algorithm}
\begin{algorithmic}[1] 
\Require Time horizon $[t_0,t_1]$, matrix tuple $(\bm{A}_t,\bm{B}_t,\bm{Q}_t)$ $\forall t\in[t_0,t_1]$ satisfying assumptions \textbf{A1}-\textbf{A2}, covariances $\bm{\Sigma}_0,\bm{\Sigma}_d$ satisfying assumption \textbf{A3}, numerical tolerance {\texttt{tol}}, maximum number of iteration {\texttt{maxIter}}\newline

\State Compute the state transition matrix \eqref{STMBlockFormAtBoundary} for the coefficient matrix \eqref{DefHamiltonianMatrix}

\State Make a random initial guess $\bm{P}_{0}\in\mathbb{S}^{n}$ \Comment{i.i.d. uniform $n(n+1)/2$ entries over some bounding box}

\State {\texttt{err}} $\gets 0$ \Comment{Initialize error}

\State {\texttt{idx}} $\gets 1$ \Comment{Initialize recursion index}

\While{(({\texttt{err}}$>${\texttt{tol}})\&\&({\texttt{idx}}$\leq${\texttt{maxIter}}))}

\State $\bm{H}_{0}\gets \bm{\Sigma}_{0}^{-1} - \bm{P}_{0}$ \Comment{From \eqref{HmatrixRiccatiICFinal}}

\State $\bm{H}_{1}\!\gets\!\left(\!\bm{H}_{0}\bm{\Phi}_{12}^{\!\top}-\bm{\Phi}_{11}^{\!\top}\!\right)^{\!-1}\!\!\left(\bm{\Phi}_{21}^{\!\top}\!-\!\bm{H}_{0}\bm{\Phi}_{22}^{\!\top}\right)\!$\!
\Comment{From \eqref{H0toH1asLFT}}

\State $\bm{P}_{1} \gets -\frac{\bm{H}_{1}+\bm{\Sigma}_{d}}{2} + \left(\left(\frac{\bm{H}_{1}-\bm{\Sigma}_{d}}{2}\right)^{\!\!2} + \bm{I}\right)^{\frac{1}{2}}$ \Comment{From \eqref{P1explicit}}

\State $\!\left(\bm{P}_{0}\right)_{\text{next}}\!\!\gets\! \!\left(\!\bm{P}_{1}\bm{\Phi}_{12} - \bm{\Phi}_{22}\!\right)^{\!-1}\!\!\left(\!\bm{\Phi}_{21}-\bm{P}_{1}\bm{\Phi}_{11}\!\right)$\!\Comment{From \eqref{P1toP0asLFT}}

\State {\texttt{err}} $\gets \|\left(\bm{P}_{0}\right)_{\text{next}} - \bm{P}_{0}\|_{\mathrm{Frobenius}}$ \Comment{Update error}

\State {\texttt{idx}} $\gets {\texttt{idx}} + 1$ \Comment{Update recursion index}

\State $\bm{P}_{0}\gets\left(\bm{P}_{0}\right)_{\text{next}}$

\EndWhile

\State $\bm{P}_{t}^{\mathrm{opt}}\:\forall t\in[t_0,t_1]\gets$ integrate \eqref{CovarianceCostateODE} with converged $\bm{P}_0$ as initial condition

\State $\bm{\Sigma}_{t}^{\mathrm{opt}}\:\forall t\in[t_0,t_1]\gets$ integrate \eqref{CovarianceStateODE} with $\bm{\Sigma}_0$ as initial condition, and computed $\bm{P}_{t}^{\mathrm{opt}}$ \Comment{Optimally controlled covariance}

\State $\bm{K}_{t}^{\mathrm{opt}} \gets -\bm{B}_{t}^{\top}\bm{P}_{t}^{\mathrm{opt}}\:\forall t\in[t_0,t_1]$ \Comment{Optimal feedback gain}

\State \Return $\bm{\Sigma}_{t}^{\mathrm{opt}},\bm{K}_{t}^{\mathrm{opt}} \:\forall t\in[t_0,t_1]$

\end{algorithmic}
\end{algorithm}

\subsection{Overall algorithm}
Consider the map $\bm{F}$ in Theorem \ref{thm:mainresult}. We use the converged $\bm{P}_{0}$ obtained from the recursion $\bm{P}_{0}\mapsto \bm{F}\left(\bm{P}_{0}\right)$, to solve for $\left(\bm{\Sigma}_t^{\mathrm{opt}},\bm{P}_t^{\mathrm{opt}}\right)$ from \eqref{FOOC}. Specifically, using the converged $\bm{P}_{0}$ as initial condition, we forward integrate \eqref{CovarianceCostateODE} to compute $\bm{P}_{t}^{\mathrm{opt}}$ for all $t\in[t_0,t_1]$. Using this $\bm{P}_{t}^{\mathrm{opt}}$ and the given initial condition \eqref{InitialCovariance}, we forward integrate \eqref{CovarianceStateODE} to obtain $\bm{\Sigma}_{t}^{\mathrm{opt}}$ for all $t\in[t_0,t_1]$. By Proposition \ref{prop:FOOC}, the optimal control  $\bm{u}_{t}^{\mathrm{opt}}=-\bm{B}_{t}^{\top}\bm{P}_{t}^{\mathrm{opt}}\bm{x}_t$.

For the readers' convenience, Algorithm \ref{alg:Algorithm} outlines the overall computation.

We stress here the importance of random initial guess stated in line 2 of Algorithm \ref{alg:Algorithm}. Since the convergence guarantee in Theorem \ref{thm:mainresult} holds for almost every $\bm{P}_{0}\in\mathbb{S}^{n}$, Algorithm \ref{alg:Algorithm} converges in practice (i.e., with probability one).

\section{Numerical Examples}\label{sec:examples}
In this Section, we illustrate the effectiveness of the proposed covariance steering algorithm via two numerical examples. In both, we fix $[t_0,t_1]=[0,1]$.

To implement the recursion $\bm{P}_{0}\mapsto \left(\bm{P}_{0}\right)_{\text{next}}$, we use \texttt{while} loop with convergence criterion that all components of $\bm{P}_{0}$ and $\left(\bm{P}_{0}\right)_{\text{next}}$ are within an absolute deviation (i.e., numerical tolerance) of $10^{-8}$ or less.


\subsection{Noisy double integrator}\label{subsec:DI}
\noindent We solve an instance of \eqref{OCP} for the noisy double integrator with $n=2$ states, and $m=1$ input and noise:
\begin{subequations}
\begin{align}
    \differential x_{1t} &= x_{2t}\:\differential t,\\
    \differential x_{2t} &= u_{t}\:\differential t + \differential w_t.
\end{align}
\label{NoisyDoubleIntegrator}    
\end{subequations}
The pair $\left(\bm{A}_t,\bm{B}_t\right)\equiv\left(\begin{bmatrix}
0 & 1\\
0 & 0
\end{bmatrix},\begin{bmatrix}0\\
1\end{bmatrix}\right)$ $\forall t\in[t_0,t_1]$, satisfies assumption \textbf{A1}.
We fix $\bm{Q}_{t}\equiv\bm{I}$, $\rho_0 = \mathcal{N}(\bm{0}, \bm{\Sigma}_0)$, $\rho_d = \mathcal{N}(\bm{0}, \bm{\Sigma}_d)$ with randomly generated positive definite matrices 
$$\bm{\Sigma}_0=\left[\begin{array}{cc}
4.7295 & 1.9951 \\
1.9951 & 3.6157
\end{array}\right], \; \bm{\Sigma}_d=\left[\begin{array}{cc}
1.1189 & 0.7780 \\
0.7780 & 1.7407
\end{array}\right],
$$
thereby satisfying assumptions \textbf{A2}-\textbf{A3}. We may interpret $(x_{1t},x_{2t})$ as position-velocity tuple, and $u_{t}$ as input force. 

Fig. \ref{fig:RecursionPlotDI} illustrates the convergence of the recursive algorithm $\bm{P}_{0}\mapsto \left(\bm{P}_{0}\right)_{\text{next}}$ proposed in Sec. \ref{sec:RecrsiveAlgorithm} with randomly initialized $\bm{P}_{0}\in\mathbb{S}^{2}$. The converged $\bm{P}_{0}\in\mathbb{S}^{2}$ is then used to solve for $\left(\bm{\Sigma}_t^{\mathrm{opt}},\bm{P}_t^{\mathrm{opt}}\right)$ from \eqref{FOOC}.

Fig. \ref{fig:EllipsoidPlotDI} shows the optimally controlled evolution of $1\sigma$ ellipses of $\bm{\Sigma}_{t}^{\mathrm{opt}}$ overlaid on $5$ closed-loop state sample paths w.r.t. time $t\in [0,1]$. The initial conditions for these $5$ paths are sampled from $\rho_{0}$, and the paths are computed via Euler-Maruyama integration of \eqref{NoisyDoubleIntegrator} with step size $5\times10^{-4}$.

The red ellipse at $t=1$ is the $1\sigma$ ellipse for the desired covariance $\bm{\Sigma}_d$. The optimally controlled terminal covariance
\begin{equation*}
\bm{\Sigma}_1=\left[\begin{array}{cc}
4.2282 & -0.0504 \\
-0.0504 & 1.7726
\end{array}\right],
\end{equation*}
which from Fig. \ref{fig:EllipsoidPlotDI}, is close to the desired but not identical, as expected. 

Fig. \ref{fig:ControlEffortPlot(DI)} shows the optimal input sample paths corresponding to the closed-loop sample paths in Fig. \ref{fig:EllipsoidPlotDI}.

\begin{figure}[t] 
    \centering
\includegraphics[width=0.95\linewidth]{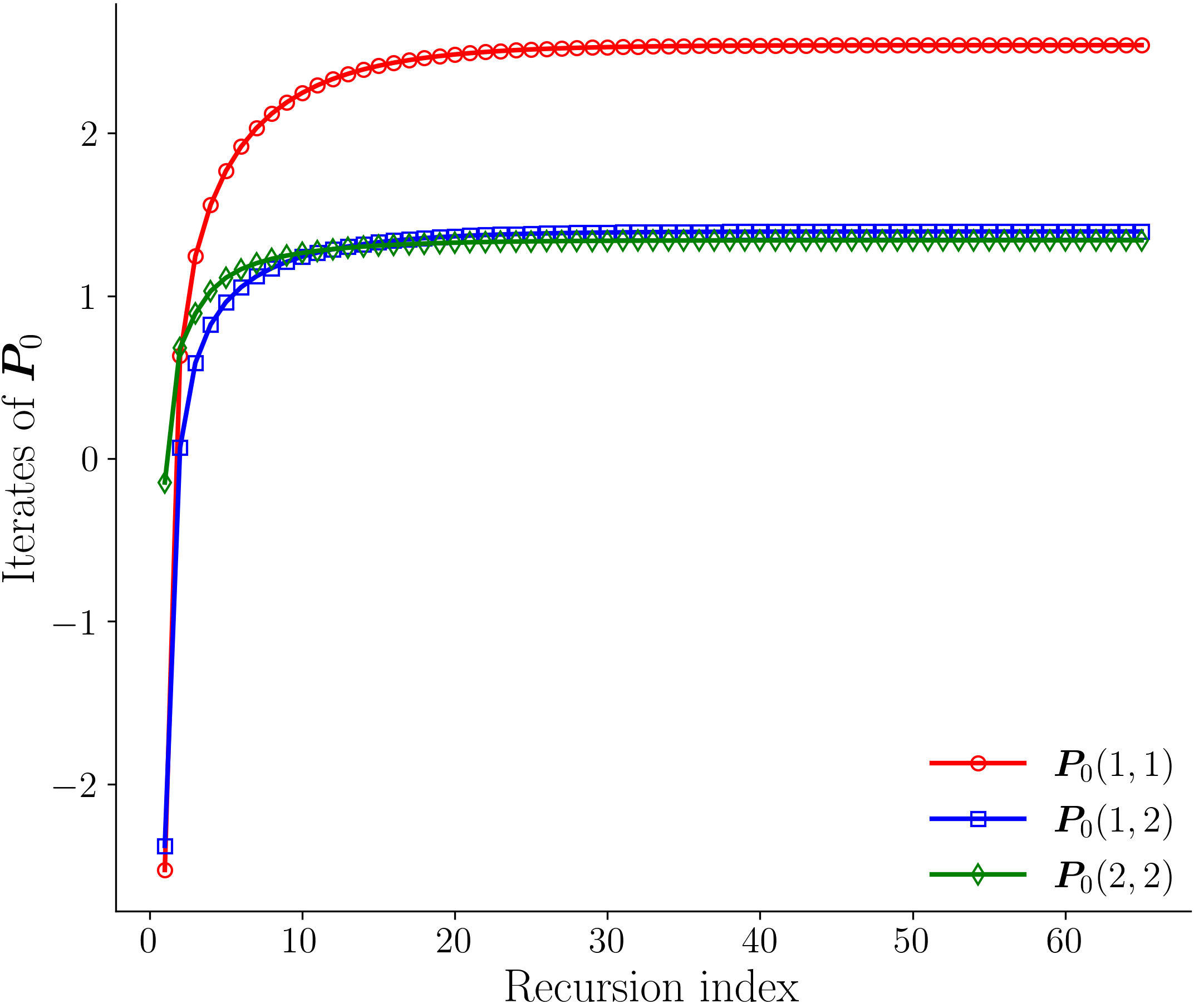}
    \caption{Convergence of the $3$ triangular elements of $\bm{P}_0\in\mathbb{S}^{2}$ for the numerical example in Sec. \ref{subsec:DI}.}
    \label{fig:RecursionPlotDI}
\end{figure}

\begin{figure}[t] 
    \centering
    \includegraphics[width=0.9\linewidth]{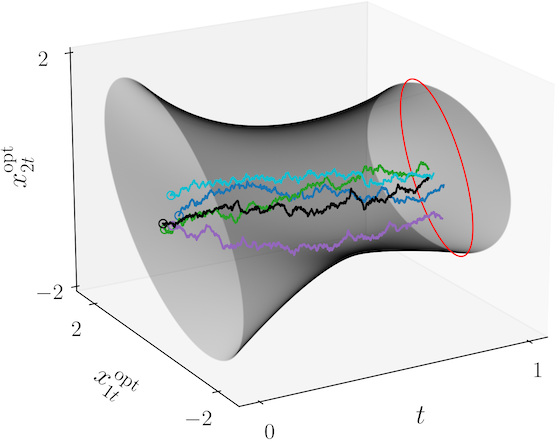}
    \caption{$500$ optimally controlled covariance snapshots (gray ellipses) and $5$ closed-loop state sample paths for the numerical example in Sec. \ref{subsec:DI}. The hollow circular markers denote the initial conditions for these sample paths. The red-lined ellipse shows the desired terminal covariance $\bm{\Sigma}_{d}$.}
    \label{fig:EllipsoidPlotDI}
\end{figure}

\begin{figure}[t] 
    \centering
    \includegraphics[width=0.95\linewidth]{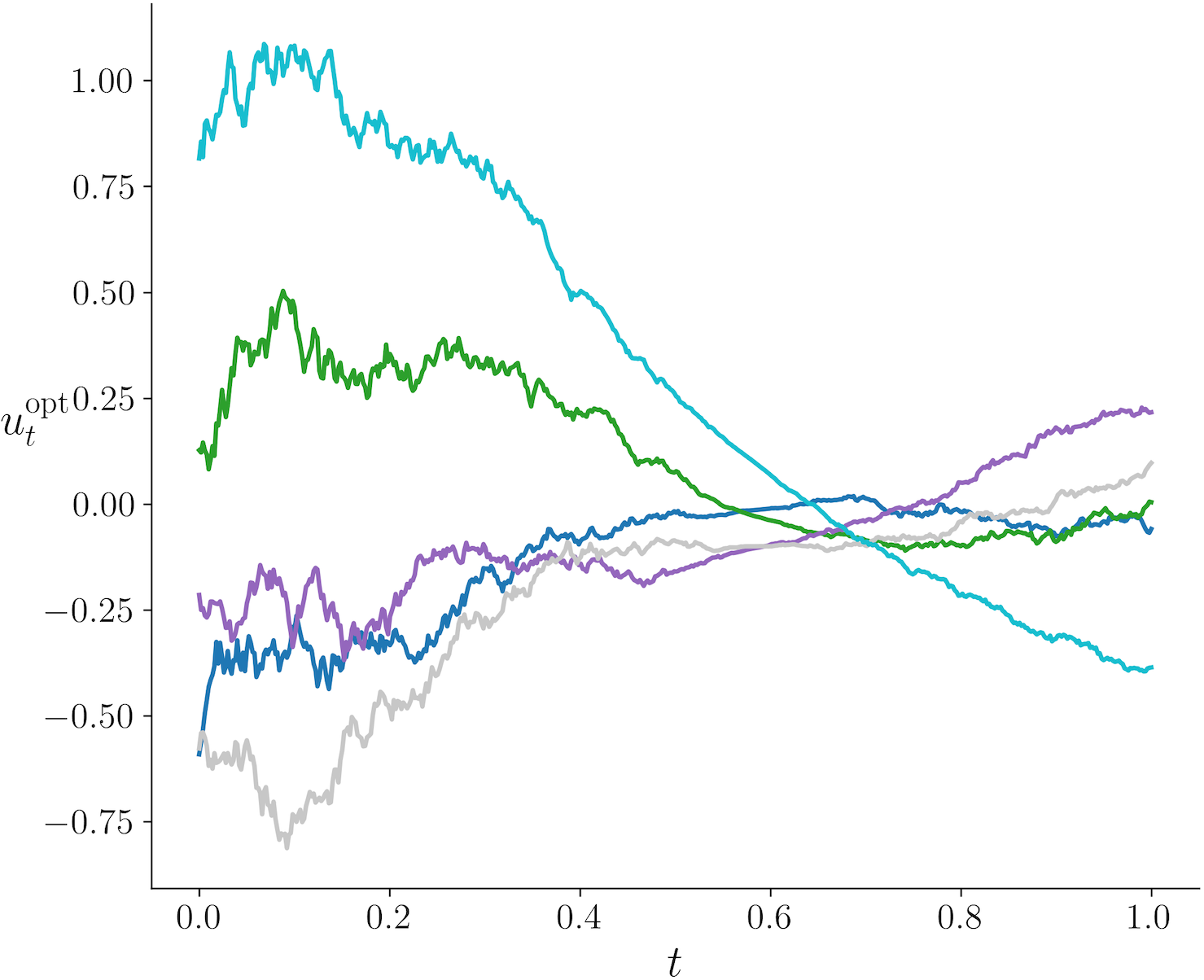}
    \caption{Optimal input paths corresponding to the state sample paths in Fig. \ref{fig:EllipsoidPlotDI}.}
    \label{fig:ControlEffortPlot(DI)}
\end{figure}


\begin{figure}[t]
\centering
\includegraphics[width=\linewidth]{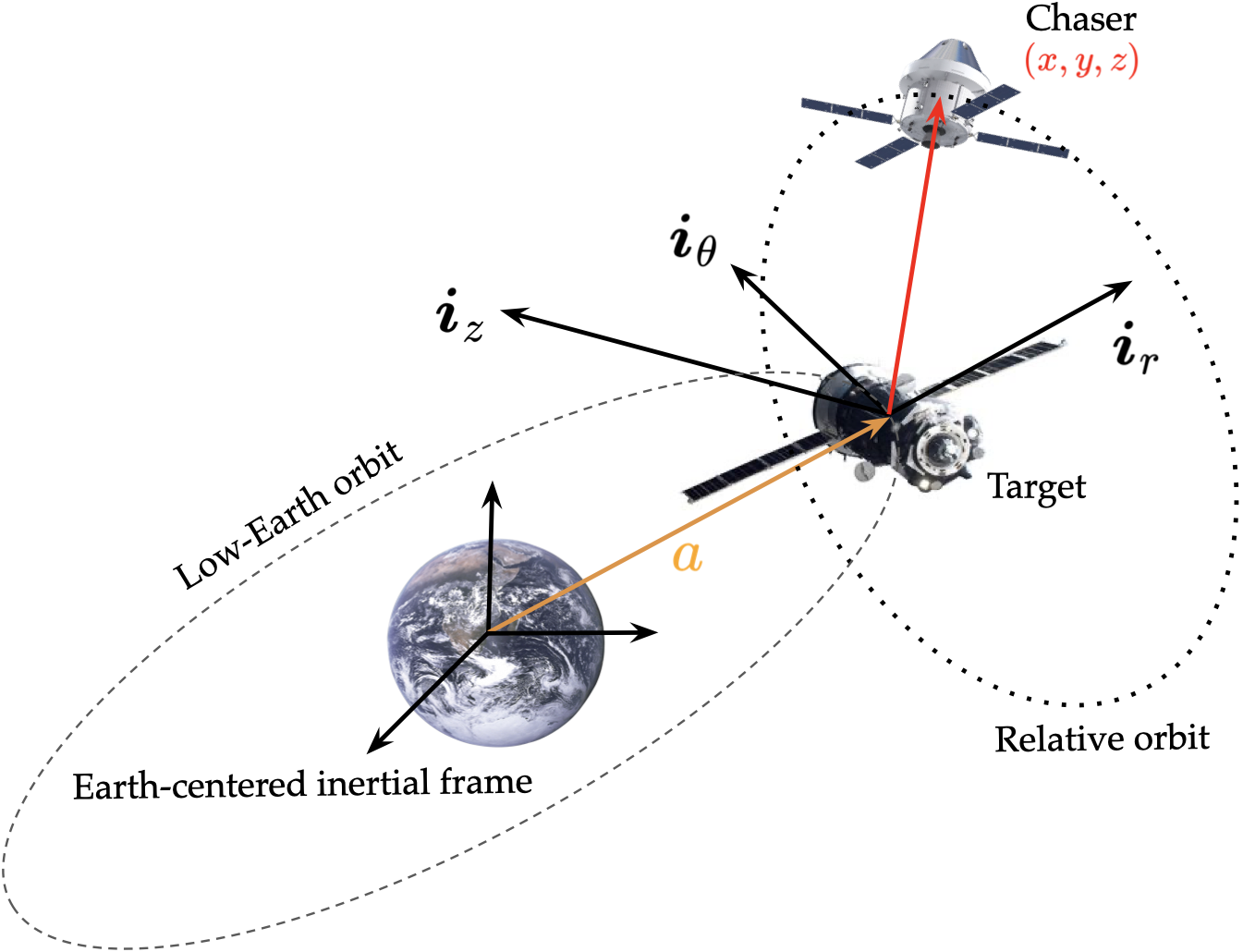}
    \caption{The relative orbital dynamics for the noisy Clohessy–Wiltshire model in Sec. \ref{subsec:CW}.}
    \label{fig:CWschematic}
\end{figure}

\begin{figure}[t] 
    \centering
    \includegraphics[width=0.95\linewidth]{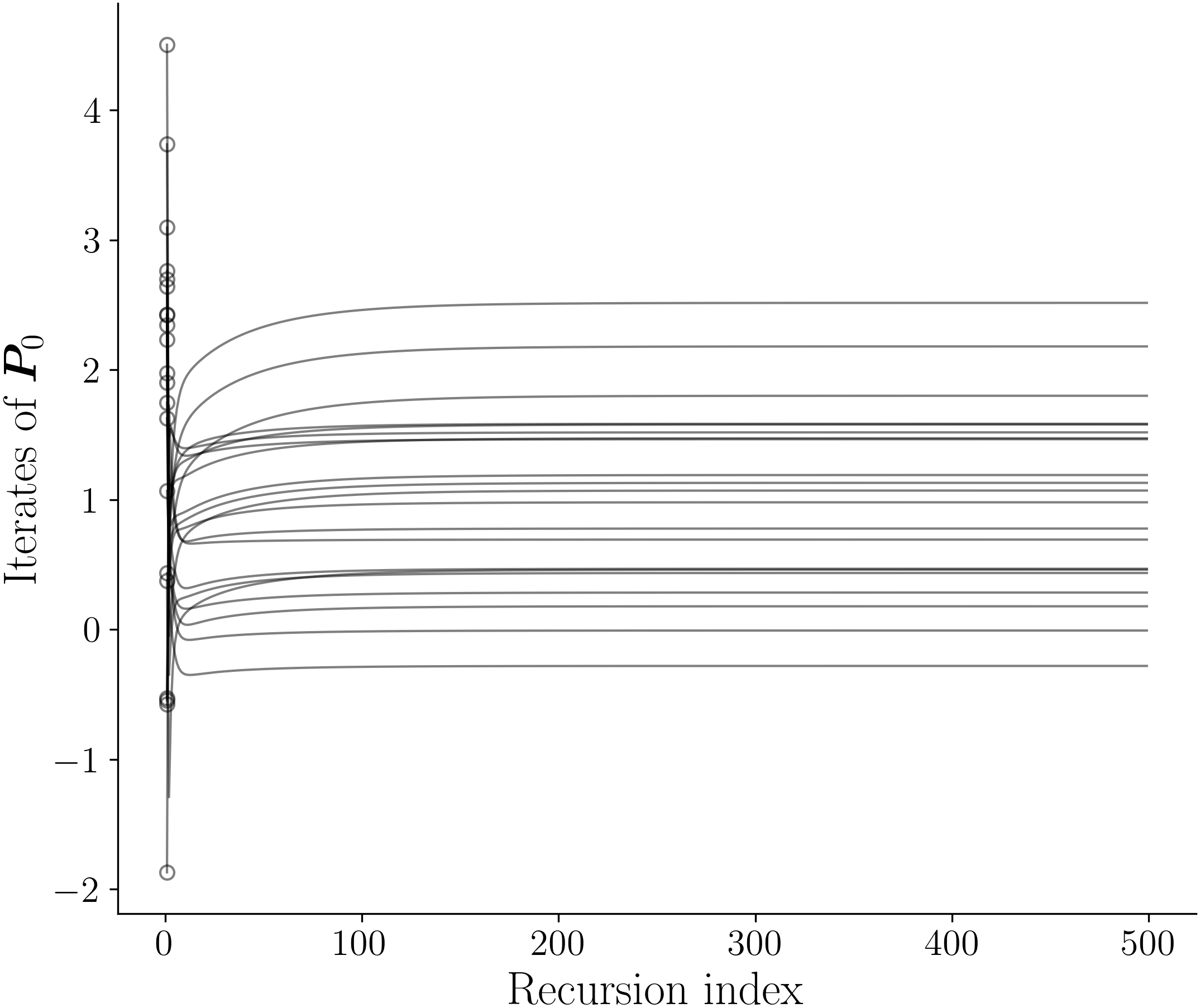}
    \caption{Convergence of the $21$ triangular elements of $\bm{P}_0\in\mathbb{S}^{6}$ for the numerical example in Sec. \ref{subsec:CW}. The random initializations are shown as hollow circular markers.}
    \label{fig:RecursionPlotCW}
\end{figure}

\begin{figure}[t!]
    \centering
    \begin{subfigure}[t]{0.8\linewidth}
        \centering
        \includegraphics[width=\linewidth]{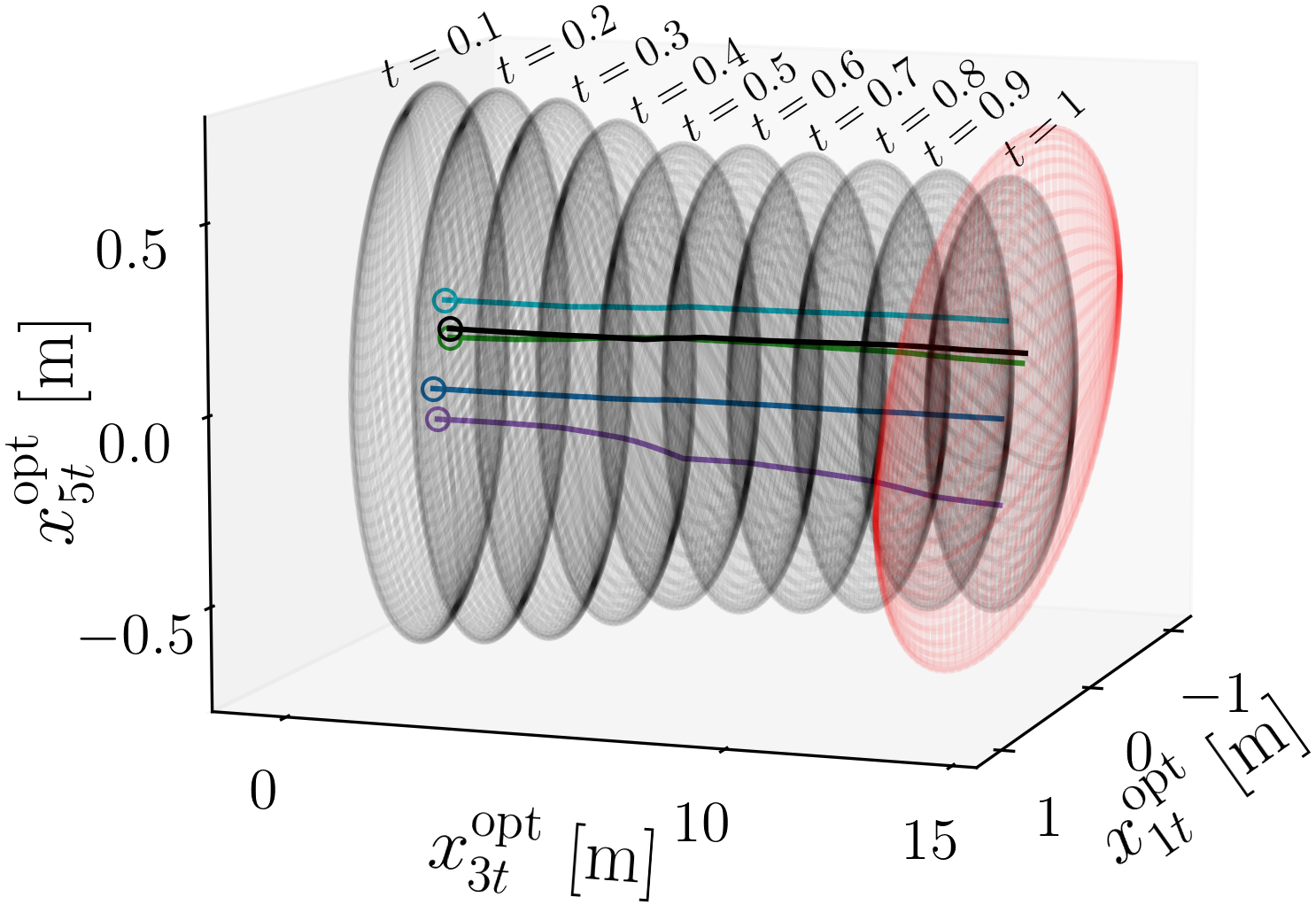}
        \caption{Marginal position covariances.}
    \end{subfigure}\\
\vspace*{0.1in} 
    \begin{subfigure}[t]{0.75\linewidth}
        \centering
        \includegraphics[width=\linewidth]{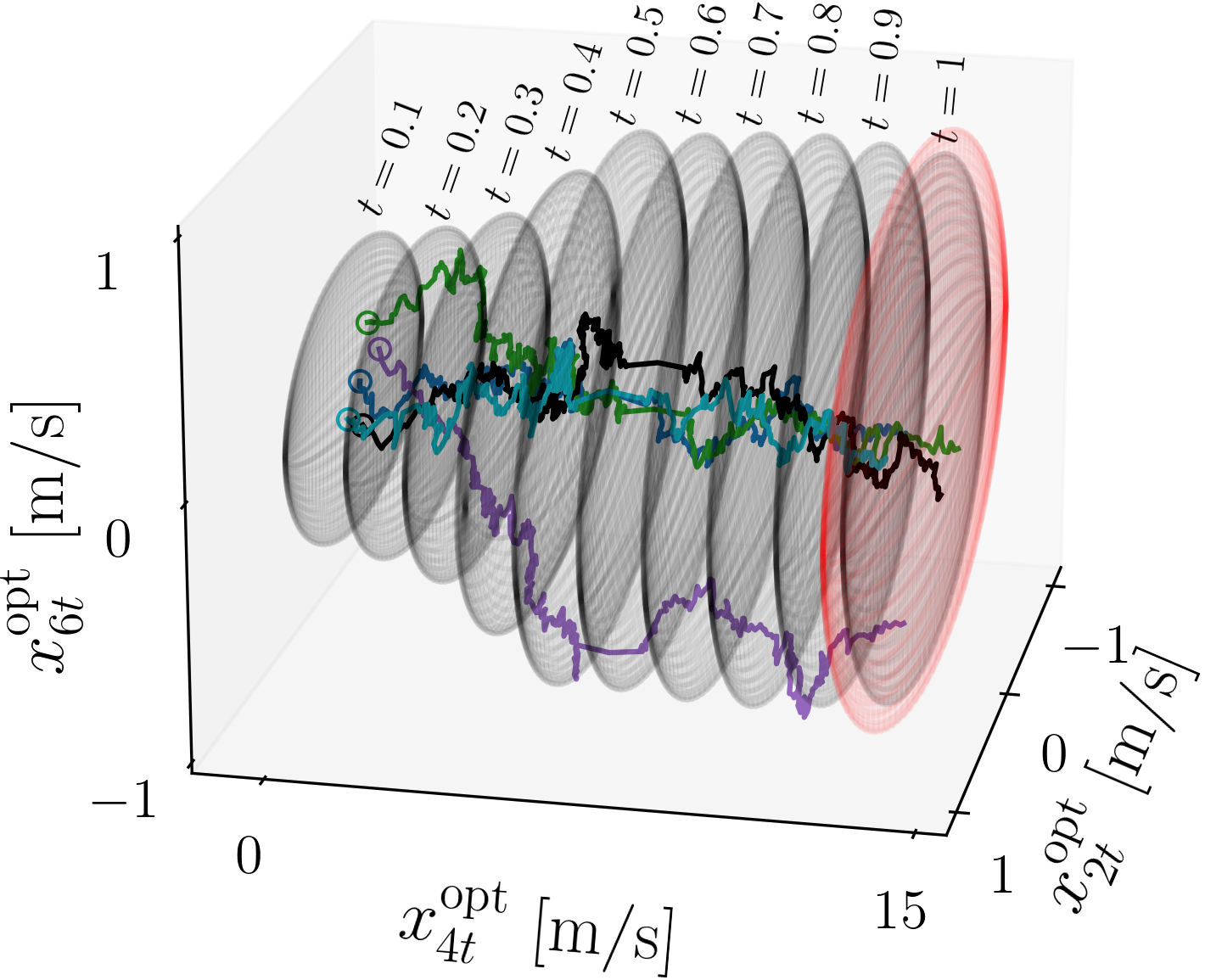}
        \caption{Marginal velocity covariances.}
    \end{subfigure}
    \caption{Optimally controlled covariances (gray wireframe ellipsoids) and $5$ closed-loop state sample paths for the numerical example in Sec. \ref{subsec:CW} in the (a) position and (b) velocity coordinates. The hollow circular markers denote the initial conditions for these sample paths. The red wireframe ellipsoids correspond to the position and velocity marginal covariances of $\bm{\Sigma}_{d}$.}
\label{CWcovariances}    
\end{figure}



\begin{figure}[t] 
    \centering
    \includegraphics[width=\linewidth]{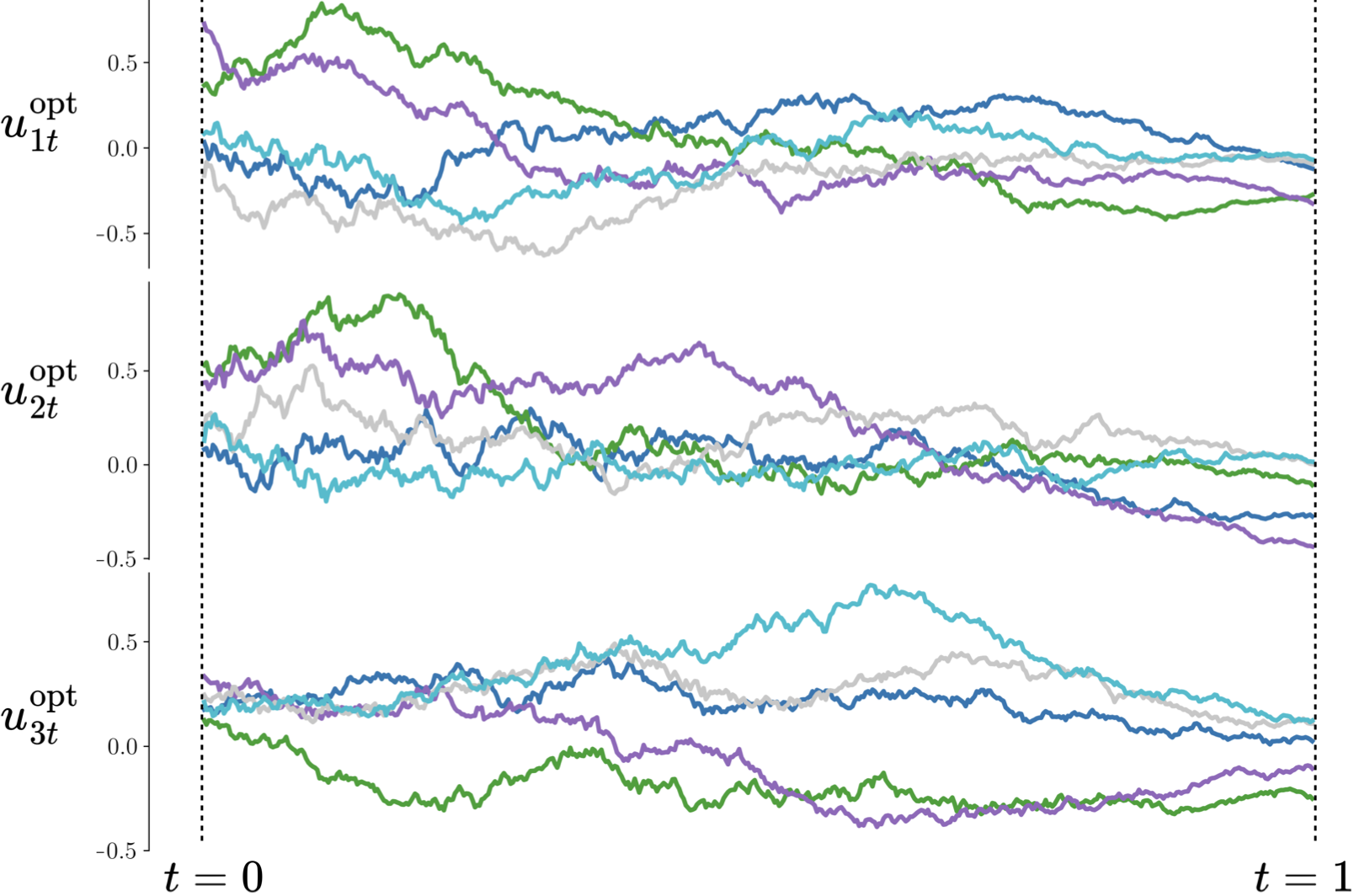}
    \caption{Sample paths for the optimal inputs (in m/s$^2$) corresponding to the state sample paths in Fig. \ref{CWcovariances}.}
    \label{fig:ControlEffortPlotCW}
\end{figure}

\subsection{Noisy Clohessy–Wiltshire model}\label{subsec:CW}

We next consider a close-proximity orbital rendezvous dynamics depicted in Fig. \ref{fig:CWschematic} where a service spacecraft (hereafter ``chaser") is dispatched\footnote{from a nearby parking orbit} to inspect a satellite (hereafter ``target") located in a circular Low-Earth Orbit (LEO) at an altitude of $415.137$ km.



As standard, we denote the relative 3D position vector of the chaser w.r.t. the target as $x\bm{i}_r + y\bm{i}_{\theta} + z\bm{i}_z$, where the unit vector $\bm{i}_r$ points in radially outward direction from the target satellite, $\bm{i}_{\theta}$ is along the direction of motion of the target satellite, and $\bm{i}_z$ is normal to the orbital plane. 

The noisy version of the Clohessy-Wiltshire equation \cite{clohessy1960terminal} models the linearized relative dynamics of the chaser w.r.t. the target. This model has $n=6$ states, $m=3$ inputs and noises, and is the following instance of \eqref{OCPdynamics}:
\begin{subequations}
\begin{align}
    \differential x_{1t} &= x_{2t}\:\differential t,\\
    \differential x_{2t} &= \left(3\nu^{2}x_{1t} + 2\nu x_{4t} + u_{1t}\right)\differential t + \differential w_{1t},\\
    \differential x_{3t} &= x_{4t}\:\differential t,\\
    \differential x_{4t} &= \left(-2\nu x_{2t} + u_{2t}\right)\differential t + \differential w_{2t},\\
    \differential x_{5t} &= x_{6t}\:\differential t,\\
    \differential x_{6t} &= \left(-\nu^{2} x_{5t} + u_{3t}\right)\differential t + \differential w_{3t},
\end{align}
\label{NoisyCW}    
\end{subequations}
where $(x_{1t},x_{3t},x_{5t})\equiv(x,y,z)$ are the relative position coordinates explained earlier, $(x_{2t},x_{4t},x_{6t})$ are the relative velocity coordinates, and $(u_{1t},u_{2t},u_{3t})$ are the thruster acceleration inputs. The standard Wiener processes $(w_{1t}, w_{2t}, w_{3t})$ in \eqref{NoisyCW} model process noise such as thruster actuation noise, solar radiation pressure, along the acceleration channels.

In \eqref{NoisyCW}, the parameter $\nu := \sqrt{\mu/a^{3}} = 1.1276\times10^{-3}$ rad/s is the constant orbital rate of the target, $\mu = 3.9860 \times 10^{14} \text{ m}^3/\text{s}^2$ is the standard gravitational parameter of the Earth, and $a = 6793.237$ Km is the radius of the circular LEO in which the target resides.

Let $$\bm{A}^{(1)} := \nu^2\left[3\bm{e}_{1}\vert\bm{0}_{3\times 1}\vert-\bm{e}_3\right],\quad \bm{A}^{(2)}:=\nu\left[-2\bm{e}_{2}\vert 2\bm{e}_{3}\vert\bm{0}_{3\times 1}\right],$$ where $\{\bm{e}_{i}\}_{i=1,2,3}$ are the standard basis column vectors in $\mathbb{R}^{3}$. It is easy to verify that the system matrix pair for \eqref{NoisyCW}: 
\begin{align*}
\left(\bm{A}_t,\bm{B}_t\right)\equiv\left(\begin{bmatrix}
\bm{0}_{3\times 3} & \bm{I}_{3\times 3}\\
\bm{A}^{(1)} & \bm{A}^{(2)}
\end{bmatrix},\begin{bmatrix}
\bm{0}_{3\times 3}\\
\bm{I}_{3\times 3}
\end{bmatrix}\right)\quad\forall t\in[t_0,t_1],    
\end{align*}
satisfies assumption \textbf{A1}.
 
Since covariance steering is decoupled from the mean steering (see Remark \ref{Remark:NonzeroMean}), we focus on steering the initial joint state PDF $\rho_0 = \mathcal{N}(\bm{0}, \bm{\Sigma}_0)$ close to the desired joint state PDF $\rho_d = \mathcal{N}(\bm{0}, \bm{\Sigma}_d)$. We fix $\bm{Q}=\bm{I}$, and use randomly generated positive definite 
{\small{
\begin{align*}
\bm{\Sigma}_0 &= 
\begin{bmatrix}
5.9148 & 3.8100 & 2.5815 & 2.1795 & 4.1628 & 1.9270 \\
3.8100 & 5.5664 & 2.8501 & 2.1819 & 3.8496 & 3.3638 \\
2.5815 & 2.8501 & 3.3834 & 1.5591 & 2.5389 & 2.3088 \\
2.1795 & 2.1819 & 1.5591 & 3.5850 & 2.6187 & 2.0098 \\
4.1628 & 3.8496 & 2.5389 & 2.6187 & 5.1285 & 2.5639 \\
1.9270 & 3.3638 & 2.3088 & 2.0098 & 2.5639 & 5.4354
\end{bmatrix}, \\[1em]
\bm{\Sigma}_d &= 
\begin{bmatrix}
1.6431 & 1.1138 & 1.5453 & 1.1729 & 1.2916 & 0.4077 \\
1.1138 & 1.9581 & 1.4418 & 1.0926 & 1.2408 & 0.4495 \\
1.5453 & 1.4418 & 3.9142 & 1.9928 & 2.0221 & 1.5553 \\
1.1729 & 1.0926 & 1.9928 & 2.1027 & 1.3448 & 0.9645 \\
1.2916 & 1.2408 & 2.0221 & 1.3448 & 1.7077 & 0.7830 \\
0.4077 & 0.4495 & 1.5553 & 0.9645 & 0.7830 & 1.5008
\end{bmatrix},
\end{align*}
}}
thereby satisfying assumptions \textbf{A2}-\textbf{A3}.

Fig. \ref{fig:RecursionPlotCW} shows convergence of the recursion $\bm{P}_{0}\mapsto \left(\bm{P}_{0}\right)_{\text{next}}$ proposed in Sec. \ref{sec:RecrsiveAlgorithm} with randomly initialized $\bm{P}_{0}\in\mathbb{S}^{6}$.

Fig. \ref{CWcovariances} plots the marginal $1\sigma$ covariance ellipsoids in the position (Fig. \ref{CWcovariances}(a)) and velocity (Fig. \ref{CWcovariances}(b)) coordinates\footnote{Recall that the covariance of the marginal of a Gaussian joint distribution is simply the corresponding submatrix of the original covariance matrix.}, along with $5$ closed-loop sample paths w.r.t. time $t\in[0,1]$. As in the previous example, the initial conditions for these $5$ paths are sampled from the six dimensional Gaussian $\rho_0$, and the paths are obtained by the Euler-Maruyama integration with the same step size as before. In Fig. \ref{CWcovariances}, the closed-loop position trajectories have more fluctuations than the closed-loop velocity trajectories, which is expected since \eqref{NoisyCW} is a degenerate It\^{o} diffusion.

In Fig. \ref{CWcovariances}, we indicate the directionality of time by spatially translating the centers of the covariance ellipsoids along the $y$ position and velocity coordinate, respectively. In each subfigure, an ellipsoid represents a specific covariance snapshot. 

The terminal covariance resulting from our controller is
{\small{
\begin{align*}
\bm{\Sigma}_{1} = \begin{bmatrix}
4.4809 & 3.1131 & 2.3911 & 0.4248 & 0.5287 & 0.0363 \\
3.1131 & 5.0291 & 3.0649 & 0.2636 & 0.4523 & 0.0130 \\
2.3911 & 3.0649 & 5.1735 & 1.3626 & 1.0944 & 1.2781 \\
0.4248 & 0.2636 & 1.3626 & 2.3281 & 1.2304 & 1.0207 \\
0.5287 & 0.4523 & 1.0944 & 1.2304 & 1.8847 & 0.8106 \\
0.0363 & 0.0130 & 1.2781 & 1.0207 & 0.8106 & 1.7013
\end{bmatrix}
\end{align*}
}}
\noindent which is close to the desired, as seen in Fig. \ref{CWcovariances}. 

Fig. \ref{fig:ControlEffortPlotCW} shows the optimal input sample paths corresponding to the closed-loop state sample paths in Fig. \ref{CWcovariances}. 

\section{Concluding Remarks}\label{sec:conclusions}
We formulated and solved the fixed horizon linear quadratic covariance steering problem in continuous time with a terminal cost measured in Hilbert-Schmidt (i.e., Frobenius) norm error between the desired and the controlled terminal covariances. Specifically, we derived the conditions of optimality, and transformed the resulting system of equations as a coupled system of Riccati-Riccati matrix ODEs. Using the LFTs associated with the solution maps for these ODEs and their properties, we proposed a novel matrix-valued nonlinear recursion on the space of symmetric matrices to solve the conditions of optimality. We then proved that the proposed recursion is convergent to its unique fixed point. We illustrated the computation for the optimal control and the optimally controlled covariance using two numerical examples: noisy double integrator and noisy Clohessy–Wiltshire model for close-proximity orbital rendezvous. Numerical results showed fast convergence in practice, demonstrating the practical feasibility of the proposed computation.

It is straightforward to generalize the problem formulation and the proposed solution when the terminal cost \eqref{DefTerminalCost} is weighted, i.e., of the form $\frac{1}{2}\|\bm{W}^{\frac{1}{2}}\left(\bm{\Sigma}_1-\bm{\Sigma}_{d}\right)\|_{\mathrm{Frobenius}}^{2}$ for a fixed positive definite weight matrix $\bm{W}$. Our developments also apply when the Gaussian PDFs $\rho_0,\rho_{d}$ have nonzero means, as explained in Remark \ref{Remark:NonzeroMean}. A more interesting research direction is to investigate if similar matricial recursive algorithms with convergence guarantees can be designed for other terminal costs such as the Bures-Wasserstein metric \cite{halder2016finite}, or the Fisher-Rao metric \cite[eq. (21)]{halder2018gradient} between the controlled and the desired terminal covariances. This will be pursued in our future work.


\appendix

\subsection{Auxiliary Lemmas}\label{App:LemmaF4bound}
The following lemmas are used in the proof of Theorem \ref{thm:mainresult}. The proof of Lemma \ref{Lemma:F4bound} and Lemma \ref{Lemma:NondegenWideBlocks} rely on the properties of the blocks of \eqref{STMBlockFormAtBoundary} listed in Lemma \ref{Lemma:STMblockproperties}.

\begin{lemma}[$\bm{F}_{4}$ is bounded]\label{Lemma:F4bound}
Given the state transition matrix \eqref{defSTMofM} associated with \eqref{DefHamiltonianMatrix}, consider the mapping $\bm{X}\to \bm{F}_{4}(\bm{X})$ given by \eqref{defF4} in variable $\bm{X}\in\mathbb{S}^{n}$. Assume that $\bm{X}\bm{\Phi}_{12} - \bm{\Phi}_{22}$ is invertible (i.e., $\bm{F}_{4}$ is well-defined). Then the mapping $\bm{F}_{4}$ is bounded in the sense
\begin{align}
\|{\mathrm{vec}}\left(\bm{F}_{4}(\bm{X})\right)\|_2 < \infty \quad\forall\bm{X}\in\mathbb{S}^{n},   
\end{align}
where ${\mathrm{vec}}$ denotes the vectorization operator.
\end{lemma}
\begin{proof}
For matrices $\bm{P},\bm{Q},\bm{R}$ of suitable sizes, recall the identity
\begin{align}
{\mathrm{vec}}\left(\bm{PQR}\right)=\left(\bm{R}^{\top}\otimes\bm{P}\right){\mathrm{vec}}\left(\bm{Q}\right).
\label{vecTripleProduct}  
\end{align}
Applying \eqref{vecTripleProduct}, we have
\begin{align}
&{\mathrm{vec}}\left(\bm{F}_{4}(\bm{X})\right)\nonumber\\
=& \left(\left(\bm{\Phi}_{21}-\bm{X} \bm{\Phi}_{11}\right)^{\top} \otimes\left(\bm{X} \bm{\Phi}_{12}-\bm{\Phi}_{22}\right)^{-1}\right) \operatorname{vec}\left(\bm{I}\right), 
\label{vec2kron}    
\end{align}
and therefore,
\begin{align}
&\|{\mathrm{vec}}\left(\bm{F}_{4}(\bm{X})\right)\|_2\nonumber\\
\leq &\sqrt{n}\|\left(\bm{\Phi}_{21}-\bm{X} \bm{\Phi}_{11}\right)^{\top} \otimes\left(\bm{X} \bm{\Phi}_{12}-\bm{\Phi}_{22}\right)^{-1}\|_2.
\label{vec2norm}    
\end{align}

From Lemma \ref{Lemma:STMblockproperties}, $\bm{\Phi}_{11},\bm{\Phi}_{12}$ are invertible, and we can write
\begin{align}
&\left(\bm{\Phi}_{21}-\bm{X} \bm{\Phi}_{11}\right)^{\top} \otimes\left(\bm{X} \bm{\Phi}_{12}-\bm{\Phi}_{22}\right)^{-1}\nonumber\\
=& \left(\bm{\Phi}_{11}^{\top} \otimes \bm{\Phi}_{12}^{-1}\right)\bigg\{\!\!\left(\left(\bm{\Phi}_{21} \bm{\Phi}_{11}^{-1}\right)^{\top}-\bm{X}\right) \otimes\left(\bm{X}-\bm{\Phi}_{22} \bm{\Phi}_{12}^{-1}\right)\!\!\bigg\}.
\label{PeelingOff}
\end{align}
Combining \eqref{vec2norm} with \eqref{PeelingOff}, using the sub-multiplicativity of
the induced $2$ norm, and that the norm of a Kronecker product is equal to the product of the norms, we get
\begin{align}
&\|{\mathrm{vec}}\left(\bm{F}_{4}(\bm{X})\right)\|_2
\leq\sqrt{n}\:\|\bm{\Phi}_{11}\|_{2} \|\bm{\Phi}_{12}^{-1}\|_{2}\|\bm{\Phi}_{21}\bm{\Phi}_{11}^{-1}-\bm{X}\|_2 \nonumber\\
&\qquad\qquad\qquad\qquad\qquad\|\left(\bm{X}-\bm{\Phi}_{22}\bm{\Phi}_{12}^{-1}\right)^{-1}\|_2, 
\label{BoundInitial}    
\end{align}
where we have dropped the transpose from $\bm{\Phi}_{21}\bm{\Phi}_{11}^{-1}$ since it is symmetric (from \eqref{identity3} in Lemma \ref{Lemma:STMblockproperties}).

Recall again from Lemma \ref{Lemma:STMblockproperties} that both $\bm{\Phi}_{11},\bm{\Phi}_{12}$ are invertible. Post-multiplying both sides of \eqref{identity1} in Lemma \ref{Lemma:STMblockproperties} with $\bm{\Phi}_{12}^{-1}$, and pre-multiplying the same with $\bm{\Phi}_{11}^{-\top}$, gives
$$\bm{\Phi}_{22}\bm{\Phi}_{12}^{-1}-\bm{\Phi}_{11}^{-\top}\bm{\Phi}_{21}^{\top}=\bm{\Phi}_{11}^{-\top}\bm{\Phi}_{12}^{-1}.$$
Transposing both sides of the above, we find
\begin{align}
\bm{\Phi}_{22}\bm{\Phi}_{12}^{-1}-\bm{\Phi}_{21}\bm{\Phi}_{11}^{-1} = \left(\bm{\Phi}_{12}\bm{\Phi}_{11}^{\!\top}\right)^{-1},
\label{intermedTransposition}    
\end{align}
where we used that $\bm{\Phi}_{22}\bm{\Phi}_{12}^{-1}$ is symmetric (from \eqref{identity2} in Lemma \ref{Lemma:STMblockproperties}), that $\bm{\Phi}_{21}\bm{\Phi}_{11}^{-1}$ is symmetric (from \eqref{identity3} in Lemma \ref{Lemma:STMblockproperties}), and that $\bm{\Phi}_{12}\bm{\Phi}_{11}^{\top}$ is symmetric (from \eqref{identity5} in Lemma \ref{Lemma:STMblockproperties}). From \eqref{intermedTransposition}, we then have
\begin{align}
\bm{X}-\bm{\Phi}_{22}\bm{\Phi}_{12}^{-1} &= - \left(\bm{\Phi}_{21}\bm{\Phi}_{11}^{-1}-\bm{X}\right) - \left(\bm{\Phi}_{12}\bm{\Phi}_{11}^{\!\top}\right)^{-1}\nonumber\\
&= - \bm{Y} - \bm{\Gamma},
\label{InverseFactor}    
\end{align}
where $\bm{Y}:= \bm{\Phi}_{21}\bm{\Phi}_{11}^{-1}-\bm{X}$, and $\bm{\Gamma}:= \left(\bm{\Phi}_{12}\bm{\Phi}_{11}^{\!\top}\right)^{-1}$. Notice that both $\bm{Y}$ and $\bm{\Gamma}$ are symmetric. Using \eqref{InverseFactor}, we rewrite \eqref{BoundInitial} as
\begin{align}
{\|\mathrm{vec}}\left(\bm{F}_{4}(\bm{X})\right)\|_2
\leq \sqrt{n}\:\|\bm{\Phi}_{11}\|_{2} \:\|\bm{\Phi}_{12}^{-1}\|_{2}\:\dfrac{\sigma_{\max}\left(\bm{Y}\right)}{\sigma_{\min}\left(\bm{Y}+\bm{\Gamma}\right)}, 
\label{BoundIntermed}    
\end{align}
where $\sigma_{\max},\sigma_{\min}$ denote the maximum and the minimum singular value, respectively.

From Weyl's singular value inequality:
\begin{align}
\vert\sigma_{\min}(\bm{Y}+\bm{\Gamma})-\sigma_{\min}(\bm{\Gamma})\vert \leq \sigma_{\max}(\bm{Y}),
\label{WeylIneq}    
\end{align}
we have $\sigma_{\min}(\bm{Y}+\bm{\Gamma}) \geq \max\{0,\sigma_{\min}(\bm{\Gamma}) - \sigma_{\max}(\bm{Y})\}$. On the other hand, since $\bm{X}\bm{\Phi}_{12} - \bm{\Phi}_{22}$ is invertible (per assumption), so is $\bm{X}-\bm{\Phi}_{22}\bm{\Phi}_{12}^{-1}$. Hence from \eqref{InverseFactor}, the symmetric matrix $\bm{Y}+\bm{\Gamma}$ is invertible, and $\sigma_{\min}\left(\bm{Y}+\bm{\Gamma}\right) > 0$. Putting these together with \eqref{BoundIntermed}, the statement follows since the blocks of \eqref{STMBlockFormAtBoundary} have finite norms.
\end{proof}

\begin{lemma}[$\bm{F}_{3}$ is nonexpansive]\label{Lemma:F3isNonexpansive}
For fixed $\bm{\Sigma}_{d}\succ\bm{0}$, the map $\bm{X}\to\bm{F}_{3}(\bm{X})$ given by \eqref{defF3} is nonexpansive (i.e., 1-Lipschitz) for $\bm{X}\in\mathbb{S}^{n}$.
\end{lemma}
\begin{proof}
Let ${\mathrm{Lip}}$ denote the Lipchitz constant. Writing $\bm{F}_{3} = \bm{F}_{33}\circ\bm{F}_{32}\circ\bm{F}_{31}$, where
\begin{subequations}
\begin{align}
\bm{F}_{31}(\bm{X})&:= \dfrac{\bm{X}-\bm{\Sigma}_{d}}{2},\\
\bm{F}_{32}(\bm{X})&:= -\bm{X} + \left(\bm{X}^{2}+\bm{I}\right)^{\frac{1}{2}},\\
\bm{F}_{33}(\bm{X})&:=\bm{X}-\bm{\Sigma}_{d},
\end{align}
\label{DecomposingF3}   
\end{subequations}
we note that 
\begin{align}
{\mathrm{Lip}}\left(\bm{F}_{3}\right)=\underbrace{{\mathrm{Lip}}\left(\bm{F}_{31}\right)}_{=1/2}\cdot{\mathrm{Lip}}\left(\bm{F}_{32}\right)\cdot\underbrace{{\mathrm{Lip}}\left(\bm{F}_{33}\right)}_{=1}.
\label{ProductOfLip}    
\end{align}
Since $\bm{F}_{32}:\mathbb{S}^{n}\to\mathbb{S}^{n}_{++}$ is continuously differentiable,
\begin{align}
{\mathrm{Lip}}\left(\bm{F}_{32}\right) \leq \underset{\bm{X}\in\mathbb{S}^{n}}{\sup}\|{\mathrm{D}}\bm{F}_{32}(\bm{X})\|_{2}.
\label{LipOfF32}    
\end{align}
As $\bm{X}\in\mathbb{S}^{n}$, its spectral decomposition $\bm{X}=\bm{U}{\mathrm{diag}}(\lambda_i)\bm{U}^{\top}$ where $\lambda_i\in\mathbb{R}$ are the eigenvalues of $\bm{X}$ and $\bm{U}$ is orthogonal, implies that $\bm{F}_{32}(\bm{X}) = \bm{U}{\mathrm{diag}}(\psi(\lambda_i))\bm{U}^{\top}$, where 
\begin{align}
\psi(\lambda):=-\lambda+\sqrt{\lambda^2 + 1}, \quad\lambda\in\mathbb{R}.
\label{defpsi}    
\end{align}
By Daletskii-Krein theorem \cite{daletskii1965integration},
\begin{align}
\bm{U}^{\top}\left({\mathrm{d}}\bm{F}_{32}\right)\bm{U} = \bm{\Psi} \odot \left(\bm{U}^{\top}\differential\bm{X}\bm{U}\right),
\label{DKtheorem}
\end{align}
where $\odot$ denotes the Hadamard product, the matrix $\bm{\Psi}$ has entries
\begin{align}
\Psi_{ij} := \begin{cases}
\dfrac{\psi(\lambda_i) - \psi(\lambda_{j})}{\lambda_i - \lambda_j} & \text{if}\quad \lambda_i \neq \lambda_{j},\\
\psi^{\prime}(\lambda_i) & \text{if}\quad \lambda_i = \lambda_{j},
\end{cases}
\label{defPsiMatrix}    
\end{align}
and $^{\prime}$ denotes derivative.

Applying ${\mathrm{vec}}$ to both sides of \eqref{DKtheorem}, and then using the identity \eqref{vecTripleProduct} and ${\mathrm{vec}}(\bm{A}\odot\bm{B})={\mathrm{vec}}(\bm{A})\odot{\mathrm{vec}}(\bm{B})$, we find
\begin{align}
&\left(\bm{U}^{\top}\otimes\bm{U}^{\top}\right){\mathrm{vec}}(\differential\bm{F}_{32})
={\mathrm{vec}}(\bm{\Psi})\odot\left(\bm{U}^{\top}\otimes\bm{U}^{\top}\right){\mathrm{vec}}(\differential\bm{X})\nonumber\\
&\Rightarrow{\mathrm{vec}}(\differential\bm{F}_{32}) =\left(\bm{U}\otimes\bm{U}\right){\mathrm{diag}}\left({\mathrm{vec}}(\bm{\Psi})\right)\left(\bm{U}^{\top}\otimes\bm{U}^{\top}\right)\nonumber\\
&\hspace*{2.3in}{\mathrm{vec}}(\differential\bm{X}).
\label{vecdF32}    
\end{align}
From \eqref{vecdF32}, we identify the Jacobian
\begin{align}
{\mathrm{D}}\bm{F}_{32}(\bm{X}) = \left(\bm{U}\otimes\bm{U}\right){\mathrm{diag}}\left({\mathrm{vec}}(\bm{\Psi})\right)\left(\bm{U}^{\top}\otimes\bm{U}^{\top}\right),
\label{IdentifyJacobianOfF32}    
\end{align}
and therefore, by submultiplicativity of $\|\cdot\|_2$, we obtain
\begin{align}
\|{\mathrm{D}}\bm{F}_{32}(\bm{X})\|_2 \leq \|{\mathrm{diag}}\left({\mathrm{vec}}(\bm{\Psi})\right)\|_2 = \underset{i,j}{\max}\:\vert\Psi_{ij}\vert.
\label{2normofJacobian}    
\end{align}
For any $\lambda_i,\lambda_j$ lying in an interval $I\subset\mathbb{R}$, by the mean value theorem, there exists $\theta\in [\min\{\lambda_i,\lambda_j\},\max\{\lambda_i,\lambda_j\}]\subset I$ such that $\Psi_{ij}=\psi^{\prime}(\theta)$. Hence
\begin{align}
\vert\Psi_{ij}\vert \leq \underset{\theta\in\mathbb{R}}{\sup}\:\vert\psi^{\prime}(\theta)\vert.
\label{UpperBoundAbsPsiij}    
\end{align}
\begin{figure}[t] 
    \centering
    \includegraphics[width=0.9\linewidth]{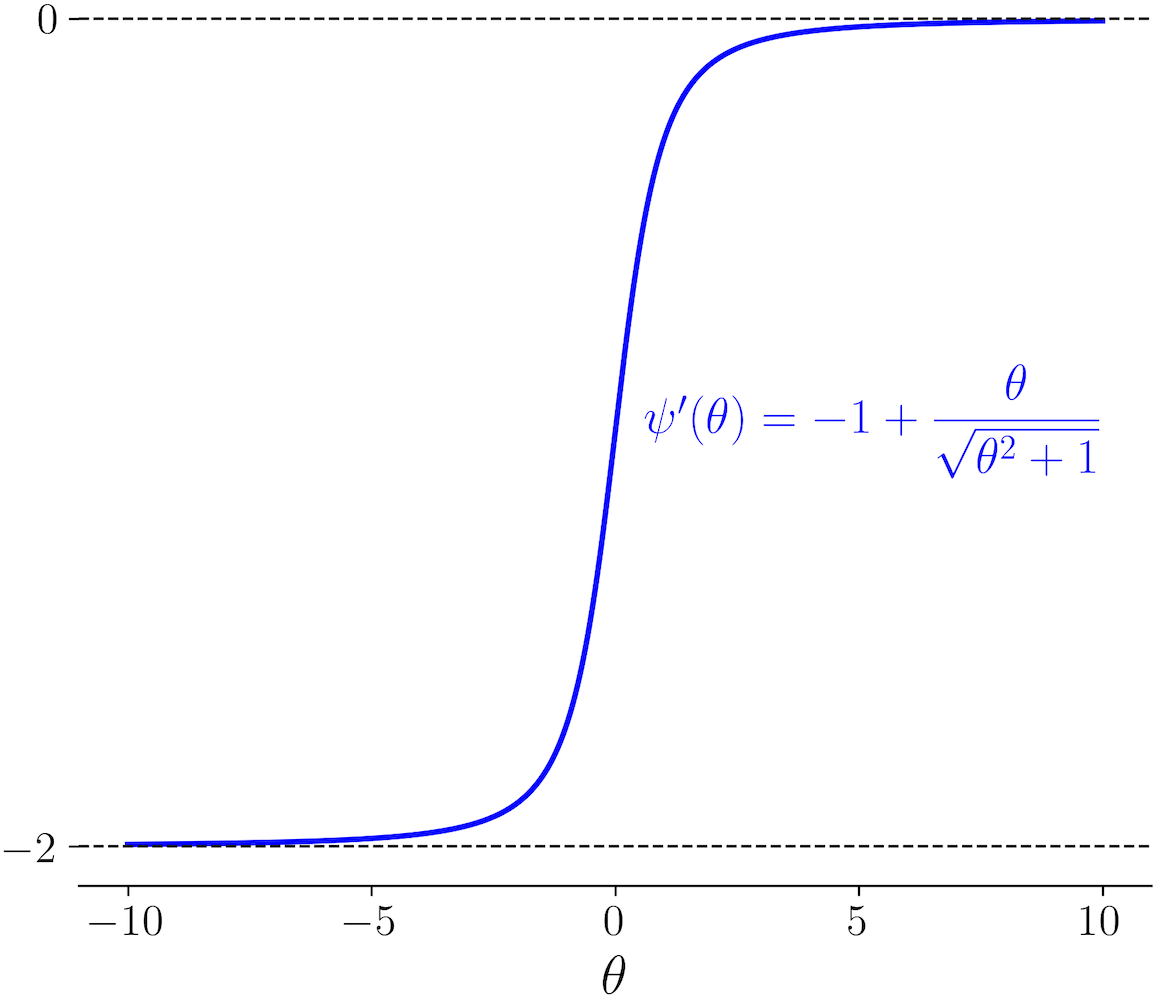}
    \caption{The function $\psi^{\prime}(\cdot)$ used in the proof of Lemma \ref{Lemma:F3isNonexpansive}.}
    \label{fig:psiprime}
\end{figure}
From \eqref{defpsi}, notice that $\psi^{\prime}(\theta) = -1 + \dfrac{\theta}{\sqrt{\theta^2 + 1}}$ is negative and strictly increasing\footnote{$\psi^{\prime\prime}(\theta) = (\theta^2 + 1)^{-3/2}>0$ $\forall\theta\in\mathbb{R}$.} over $\mathbb{R}$ (see Fig. \ref{fig:psiprime}). In particular, $\underset{\theta\in\mathbb{R}}{\sup}\:\vert\psi^{\prime}(\theta)\vert = \lim_{\theta\rightarrow-\infty}\vert\psi^{\prime}(\theta)\vert = \vert\lim_{\theta\rightarrow-\infty}\psi^{\prime}(\theta)\vert = \vert-2\vert = 2$, which combined with \eqref{2normofJacobian}-\eqref{UpperBoundAbsPsiij} yields the uniform bound 
\begin{align}
\|{\mathrm{D}}\bm{F}_{32}(\bm{X})\|_2 \leq 2.
\label{2normOfJacUpperBoundedBy2}    
\end{align}
Then, \eqref{2normOfJacUpperBoundedBy2} and \eqref{LipOfF32} together implies ${\mathrm{Lip}}\left(\bm{F}_{32}\right) \leq 2$, which further combined with \eqref{ProductOfLip}, gives
$${\mathrm{Lip}}\left(\bm{F}_{3}\right) \leq 1.$$
This completes the proof.
\end{proof}

\begin{definition}(Nondegenerate wide matrix)\label{def:NondegerateMatrix}
Given $\bm{R},\bm{S}\in\mathbb{R}^{n\times n}$, the wide matrix $\left[\bm{R},\bm{S}\right]\in\mathbb{R}^{n\times 2n}$ is said to be \emph{nondegenerate} if it has the maximal possible rank $n$. 
\end{definition}
\begin{lemma}[Nondegenerate wide blocks of \eqref{STMBlockFormAtBoundary}]\label{Lemma:NondegenWideBlocks}
Consider the state transition matrix \eqref{STMBlockFormAtBoundary} for the coefficient matrix \eqref{DefHamiltonianMatrix}. The wide matrices 
$$\left[\bm{\Phi}_{11}^{\top}, -\bm{\Phi}_{12}^{\top}\right], \; \left[\bm{\Phi}_{22}, -\bm{\Phi}_{12}\right]\in\mathbb{R}^{n\times 2n}$$
are both nondegenerate (see Definition \ref{def:NondegerateMatrix}).
\end{lemma}
\begin{proof}
Any $n\times 2n$ matrix has rank $\leq \min\{n,2n\}=n$. For the matrix $\left[\bm{\Phi}_{11}^{\top}, -\bm{\Phi}_{12}^{\top}\right]$, equality is achieved because from Lemma \ref{Lemma:STMblockproperties}, both $\bm{\Phi}_{11}, \bm{\Phi}_{12}$, and thus both $\bm{\Phi}_{11}^{\top}, -\bm{\Phi}_{12}^{\top}$ have full ranks.

To see that the aforesaid equality is also achieved for the matrix $\left[\bm{\Phi}_{22}, -\bm{\Phi}_{12}\right]$, consider any $\bm{x}=\begin{pmatrix}
    \bm{x}_1\\
    \bm{x}_2
\end{pmatrix}$ where $\bm{x}_{1},\bm{x}_2\in\mathbb{R}^{n}$, such that $\left[\bm{\Phi}_{22}, -\bm{\Phi}_{12}\right]\bm{x}=\bm{0}$. Then $\bm{x}_{2} = \bm{\Phi}_{12}^{-1}\bm{\Phi}_{22}\bm{x}_{1}$ (recall that $\bm{\Phi}_{12}$ is invertible by Lemma \ref{Lemma:STMblockproperties}). In other words, this matrix has a nullspace of dimension $n$ (parameterized by $\bm{x}_1$), and by the rank-nullity theorem \cite[p. 6]{horn2012matrix}, ${\mathrm{rank}}\left[\bm{\Phi}_{22}, -\bm{\Phi}_{12}\right]=2n-n=n$.
\end{proof}


\section*{References}

\bibliographystyle{IEEEtran}
\bibliography{References.bib}

\begin{thebibliography}{10}
\providecommand{\url}[1]{#1}
\csname url@samestyle\endcsname
\providecommand{\newblock}{\relax}
\providecommand{\bibinfo}[2]{#2}
\providecommand{\BIBentrySTDinterwordspacing}{\spaceskip=0pt\relax}
\providecommand{\BIBentryALTinterwordstretchfactor}{4}
\providecommand{\BIBentryALTinterwordspacing}{\spaceskip=\fontdimen2\font plus
\BIBentryALTinterwordstretchfactor\fontdimen3\font minus \fontdimen4\font\relax}
\providecommand{\BIBforeignlanguage}[2]{{%
\expandafter\ifx\csname l@#1\endcsname\relax
\typeout{** WARNING: IEEEtran.bst: No hyphenation pattern has been}%
\typeout{** loaded for the language `#1'. Using the pattern for}%
\typeout{** the default language instead.}%
\else
\language=\csname l@#1\endcsname
\fi
#2}}
\providecommand{\BIBdecl}{\relax}
\BIBdecl

\bibitem{halder2016finite}
A.~Halder and E.~D. Wendel, ``Finite horizon linear quadratic {G}aussian density regulator with {W}asserstein terminal cost,'' in \emph{2016 American Control Conference (ACC)}.\hskip 1em plus 0.5em minus 0.4em\relax IEEE, 2016, pp. 7249--7254.

\bibitem{bhatia2019bures}
R.~Bhatia, T.~Jain, and Y.~Lim, ``On the {B}ures--{W}asserstein distance between positive definite matrices,'' \emph{Expositiones Mathematicae}, vol.~37, no.~2, pp. 165--191, 2019.

\bibitem{givens1984class}
C.~R. Givens and R.~M. Shortt, ``A class of {W}asserstein metrics for probability distributions.'' \emph{Michigan Mathematical Journal}, vol.~31, no.~2, pp. 231--240, 1984.

\bibitem{gelbrich1990formula}
M.~Gelbrich, ``On a formula for the {L}2 {W}asserstein metric between measures on {E}uclidean and {H}ilbert spaces,'' \emph{Mathematische Nachrichten}, vol. 147, no.~1, pp. 185--203, 1990.

\bibitem{hotz1987covariance}
A.~Hotz and R.~E. Skelton, ``Covariance control theory,'' \emph{International Journal of Control}, vol.~46, no.~1, pp. 13--32, 1987.

\bibitem{zhu1995covariance}
G.~Zhu, K.~M. Grigoriadis, and R.~E. Skelton, ``Covariance control design for {H}ubble space telescope,'' \emph{Journal of Guidance, Control, and Dynamics}, vol.~18, no.~2, pp. 230--236, 1995.

\bibitem{ridderhof2020chance}
J.~Ridderhof, J.~Pilipovsky, and P.~Tsiotras, ``Chance-constrained covariance control for low-thrust minimum-fuel trajectory optimization,'' in \emph{AAS/AIAA Astrodynamics Specialist Conference}, 2020, pp. 9--13.

\bibitem{benedikter2022convex}
B.~Benedikter, A.~Zavoli, Z.~Wang, S.~Pizzurro, and E.~Cavallini, ``Convex approach to covariance control with application to stochastic low-thrust trajectory optimization,'' \emph{Journal of Guidance, Control, and Dynamics}, vol.~45, no.~11, pp. 2061--2075, 2022.

\bibitem{kumagai2024sequential}
N.~Kumagai and K.~Oguri, ``Sequential chance-constrained covariance steering for robust cislunar trajectory design under uncertainties,'' in \emph{AAS/AIAA Astrodynamics Specialist Conference}, 2024.

\bibitem{bhatia2009positive}
R.~Bhatia, \emph{Positive definite matrices}.\hskip 1em plus 0.5em minus 0.4em\relax Princeton university press, 2009.

\bibitem{balci2020covariance}
I.~M. Balci and E.~Bakolas, ``Covariance steering of discrete-time stochastic linear systems based on {W}asserstein distance terminal cost,'' \emph{IEEE Control Systems Letters}, vol.~5, no.~6, pp. 2000--2005, 2020.

\bibitem{balci2021convexity}
I.~M. Balci, A.~Halder, and E.~Bakolas, ``On the convexity of discrete time covariance steering in stochastic linear systems with {W}asserstein terminal cost,'' in \emph{2021 60th IEEE Conference on Decision and Control (CDC)}.\hskip 1em plus 0.5em minus 0.4em\relax IEEE, 2021, pp. 2318--2323.

\bibitem{yin2022trajectory}
J.~Yin, Z.~Zhang, E.~Theodorou, and P.~Tsiotras, ``Trajectory distribution control for model predictive path integral control using covariance steering,'' in \emph{2022 International Conference on Robotics and Automation (ICRA)}.\hskip 1em plus 0.5em minus 0.4em\relax IEEE, 2022, pp. 1478--1484.

\bibitem{saravanos2024distributed}
A.~D. Saravanos, I.~M. Balci, E.~Bakolas, and E.~A. Theodorou, ``Distributed model predictive covariance steering,'' in \emph{2024 IEEE/RSJ International Conference on Intelligent Robots and Systems (IROS)}.\hskip 1em plus 0.5em minus 0.4em\relax IEEE, 2024, pp. 5740--5747.

\bibitem{morimoto2024minimum}
K.~Morimoto and K.~Kashima, ``Minimum energy density steering of linear systems with {G}romov-{W}asserstein terminal cost,'' \emph{IEEE Control Systems Letters}, vol.~8, pp. 586--591, 2024.

\bibitem{nakashima2025formation}
H.~Nakashima, S.~Ganguly, K.~Morimoto, and K.~Kashima, ``Formation shape control using the {G}romov-{W}asserstein metric,'' \emph{arXiv preprint arXiv:2503.21538}, 2025.

\bibitem{memoli2011gromov}
F.~M{\'e}moli, ``Gromov--{W}asserstein distances and the metric approach to object matching,'' \emph{Foundations of computational mathematics}, vol.~11, no.~4, pp. 417--487, 2011.

\bibitem{salmona2022gromov}
A.~Salmona, J.~Delon, and A.~Desolneux, ``Gromov-{W}asserstein distances between {G}aussian distributions,'' \emph{Journal of Applied Probability}, vol.~59, no.~4, 2022.

\bibitem{hoshino2023finite}
K.~Hoshino, ``Finite-horizon optimal control of continuous-time stochastic systems with terminal cost of {W}asserstein distance,'' in \emph{2023 62nd IEEE Conference on Decision and Control (CDC)}.\hskip 1em plus 0.5em minus 0.4em\relax IEEE, 2023, pp. 5825--5830.

\bibitem{magnus2019matrix}
J.~R. Magnus and H.~Neudecker, \emph{Matrix differential calculus with applications in statistics and econometrics}.\hskip 1em plus 0.5em minus 0.4em\relax John Wiley \& Sons, 2019.

\bibitem{chen2015optimal}
Y.~Chen, T.~Georgiou, and M.~Pavon, ``Optimal steering of inertial particles diffusing anisotropically with losses,'' in \emph{2015 American Control Conference (ACC)}.\hskip 1em plus 0.5em minus 0.4em\relax IEEE, 2015, pp. 1252--1257.

\bibitem{ChenGeorgiouPavonPartIII}
Y.~Chen, T.~T. Georgiou, and M.~Pavon, ``Optimal steering of a linear stochastic system to a final probability distribution—{Part III},'' \emph{IEEE Transactions on Automatic Control}, vol.~63, no.~9, pp. 3112--3118, 2018.

\bibitem{brockett2015finite}
R.~W. Brockett, \emph{Finite dimensional linear systems}.\hskip 1em plus 0.5em minus 0.4em\relax SIAM, 2015.

\bibitem{shayman1986phase}
M.~A. Shayman, ``Phase portrait of the matrix {R}iccati equation,'' \emph{SIAM Journal on Control and Optimization}, vol.~24, no.~1, pp. 1--65, 1986.

\bibitem{potapov1988linear}
V.~Potapov, ``Linear fractional transformations of matrices,'' \emph{American Mathematical Society Translations: Series 2}, vol. 138, pp. 21--35, 1988.

\bibitem{sylvester1884equation}
J.~J. Sylvester, ``Sur l’{\'e}quation en matrices px= xq,'' \emph{CR Acad. Sci. Paris}, vol.~99, no.~2, pp. 67--71, 1884.

\bibitem{bhatia1997and}
R.~Bhatia and P.~Rosenthal, ``How and why to solve the operator equation $ax- xb= y$,'' \emph{Bulletin of the London Mathematical Society}, vol.~29, no.~1, pp. 1--21, 1997.

\bibitem{taussky1959similarity}
O.~Taussky and H.~Zassenhaus, ``On the similarity transformation between a matirx and its transpose,'' \emph{Pacific J. Math.}, vol.~9, no.~4, pp. 893--896, 1959.

\bibitem{lancaster1980existence}
P.~Lancaster and L.~Rodman, ``Existence and uniqueness theorems for the algebraic {R}iccati equation,'' \emph{International Journal of Control}, vol.~32, no.~2, pp. 285--309, 1980.

\bibitem{kuvcera1991algebraic}
V.~Ku{\v{c}}era, ``Algebraic {R}iccati equation: Hermitian and definite solutions,'' in \emph{The Riccati Equation, Eds. S. Bittanti, A.J. Laub and J.C.Willems}.\hskip 1em plus 0.5em minus 0.4em\relax Springer, 1991, pp. 53--88.

\bibitem{1099831}
J.~C. Willems, ``Least squares stationary optimal control and the algebraic {R}iccati equation,'' \emph{IEEE Transactions on Automatic Control}, vol.~16, no.~6, pp. 621--634, 1971.

\bibitem{milnor1965topology}
J.~W. Milnor, \emph{Topology from the differentiable viewpoint}.\hskip 1em plus 0.5em minus 0.4em\relax Princeton University Press, Princeton, New Jersey, 1965.

\bibitem{rudin1976principles}
W.~Rudin, \emph{Principles of mathematical analysis}, 3rd~ed.\hskip 1em plus 0.5em minus 0.4em\relax McGraw-Hill Science, 1976.

\bibitem{clohessy1960terminal}
W.~Clohessy and R.~Wiltshire, ``Terminal guidance system for satellite rendezvous,'' \emph{Journal of the aerospace sciences}, vol.~27, no.~9, pp. 653--658, 1960.

\bibitem{halder2018gradient}
A.~Halder and T.~T. Georgiou, ``Gradient flows in filtering and {F}isher-{R}ao geometry,'' in \emph{2018 Annual American Control Conference (ACC)}.\hskip 1em plus 0.5em minus 0.4em\relax IEEE, 2018, pp. 4281--4286.

\bibitem{daletskii1965integration}
J.~L. Daletskii and S.~G. Krein, ``Integration and differentiation of functions of {H}ermitian operators and applications to the theory of perturbations,'' \emph{AMS Translations (2)}, vol.~47, no. 1-30, pp. 10--1090, 1965.

\bibitem{horn2012matrix}
R.~A. Horn and C.~R. Johnson, \emph{Matrix analysis}, 2nd~ed.\hskip 1em plus 0.5em minus 0.4em\relax Cambridge university press, 2012.

\end{thebibliography}






\end{document}